\documentclass{article}
\usepackage[T1]{fontenc}
\usepackage[mathscr]{euscript}
\usepackage{amsmath}
\usepackage{amsthm}
\usepackage{stackrel, amssymb}
\usepackage{physics}
\usepackage{tensor}
\usepackage{mathtools}
\usepackage{slashed}
\usepackage{quiver}
\usepackage{hyperref}
\hypersetup{ 
	colorlinks=true,
	linkcolor=blue,
	urlcolor=blue,
	citecolor=blue
}
\usepackage{geometry}
\usepackage{sectsty}
\subsectionfont{\normalfont\itshape}
\subsubsectionfont{\normalfont\itshape}
\usepackage[toc, page]{appendix}
\usepackage{authblk}
\usepackage{xcolor}
\usepackage{biblatex}
\numberwithin{equation}{section}

\newtheorem{prop}{Proposition}
\numberwithin{prop}{section}

\numberwithin{definition}{section}

\newtheorem{theorem}{Theorem}
\numberwithin{theorem}{section}

\title{Nonabelian shift operators and shifted Yangians}
\author{Spencer Tamagni}
\affil{\textit{Center for Theoretical Physics, University of California, Berkeley}}
\date{\today}

\begin{document}
\maketitle 
\begin{abstract}
We introduce nonabelian analogs of shift operators in the enumerative theory of quasimaps. We apply them on the one hand to strengthen the emerging analogy between enumerative geometry and the geometric theory of automorphic forms, and on the other hand to obtain results about quantized Coulomb branch algebras. In particular, we find a short and direct proof that the equivariant convolution homology of the affine Grassmannian of $GL_n$ is a quotient of a shifted Yangian. 
\end{abstract}
\setcounter{tocdepth}{2}
\tableofcontents

\section{Introduction}
In this paper, we apply geometric tools to solve certain problems in representation theory. Our geometric tools allow us to give a novel identification of certain differential and difference equations that control enumerative generating functions, and as a rather basic application of them we are able to give a geometric characterization of comultiplication in quantized Coulomb algebras \cite{bfnslice}, \cite{finkelberg2017}, \cite{krylovperunov}. 

\subsection{Nonabelian shift operators}
The key new objects that we introduce and study in this paper are nonabelian analogs of shift operators. Shift operators are objects defined in enumerative geometry which provide a concrete bridge between the counting problems studied in that subject and questions of geometric representation theory. We will recall some background to explain this statement more fully, and then explain the sense in which we generalize this picture in the present paper. 

\subsubsection{}
The usual setup for studying shift operators is the following. Let $X$ be a smooth quasiprojective variety with the action of a torus $T$. In enumerative geometry one studies integrals (in cohomology or $K$-theory) over suitably compactified moduli spaces of maps $C \to X$, where $C$ is an algebraic curve that may be fixed or vary in moduli. For this paper, $C \simeq \mathbb{P}^1$ is fixed in moduli, and our preferred compactification is the moduli space of quasimaps as in \cite{cfkimmaulik} (so, in particular, we assume $X$ has a GIT quotient presentation). Then there is a group $\mathbb{C}^\times_q = \text{Aut}(C, 0, \infty)$ of automorphisms of $C$ preserving two marked points, which is important to incorporate into the problem. 

The principal objects of study in the modern theory are the \textit{vertex functions} in the terminology of \cite{okounkovpcmi}. These are the generating functions $\mathsf{Vertex}(z, a)$ of counts of rational curves in $X$ of all possible degrees. Counts are always performed in the virtual and equivariant sense; this means in particular that $\mathsf{Vertex}$ depends explicitly on variables $a$ valued in the torus $T$ (in the $K$-theoretic setting) or its Lie algebra (in the cohomological setting). It also depends on a set of variables $z$ whose role is to record the degree of the map $C \to X$; in particular there are $\text{rk} \, H^2(X; \mathbb{C})$ many $z$-variables. 

\subsubsection{}
One of the interests of the vertex functions is that they can be used to produce the fundamental solutions of differential and difference equations. Specifically, the $K$-theoretic vertex functions satisfy difference equations in all variables, and the vertex functions in cohomology satisfy differential equations in the $z$-variables and difference equations in the $a$-variables. We will focus on the latter.

Difference equations in the $a$-variables depend on a choice of cocharacter $\sigma: \mathbb{C}^\times \to T$. Given $q \in \mathbb{C}^\times_q$, a cocharacter provides on the one hand a way to shift the $a$-variables via $a \mapsto q^\sigma a$. On the other hand, we may use $\sigma$ together with the $T$-action on $X$ to define 
\begin{equation}
    \widetilde{X} = (X \times \mathbb{C}^2 \setminus \{0 \})/\,\mathbb{C}^\times
\end{equation}
which is an $X$-bundle over $C = \mathbb{P}^1$. The key observation leading to a geometric approach to difference equations is that the vertex function with shifted variables $\mathsf{Vertex}(z, q^\sigma a)$ may be interpreted analogously as a count of sections of the bundle $\widetilde{X}$; the usual $\mathsf{Vertex}(z, a)$ may indeed be viewed as the count of sections of a trivial $X$-bundle. 

\subsubsection{}
In particular, following this line of reasoning one can study the count of sections of $\widetilde{X}$ using the technique of degeneration of the base curve $C$, and this is precisely what is done in the $K$-theoretic setting in \cite{okounkovpcmi} when $X$ is a Nakajima quiver variety. This leads one to define certain operators $\mathbf{S}_\sigma$ acting in equivariant $K$-theory $K_{\mathbb{C}^\times_q \times T}(X)$ (or $H^\bullet_{\mathbb{C}^\times_q \times T}(X)$ in the cohomological setting, see \cite{MO}) which are referred to as shift operators and determine the difference equations in $a$-variables. The matrix coefficients of these operators may also be interpreted as counts of sections of $\widetilde{X}$, but the counts are defined using different moduli spaces. 

In \cite{okounkovpcmi}, the action of $\mathbf{S}_\sigma$ for minuscule $\sigma$ in the $K$-theoretic stable envelope basis is characterized via geometric $R$-matrices, and gives an explicit identification of the $q$-difference equation in the $a$-variables satisfied by $\mathsf{Vertex}(z, a)$ with the quantum Knizhnik-Zamolodchikov (qKZ) equation for a certain geometrically defined quantum group. 

\subsubsection{}
In this paper, we will take a new and different path towards identifying the difference equations, which involves certain wall-crossing formulas that we explain in Section \ref{geometricdiffeq}. Our techniques lead to a novel connection with quantum groups, which appear via the quantized Coulomb branch algebras that arise as their quotients \cite{bfnslice}. 

Geometrically, we will still be studying counts of sections of nontrivial $X$-bundles over $C$, and we will somewhat loosely refer to all such counts as ``shift operators'' acting on the vertex functions. The sense in which this is reasonable should become clear from context, but the reader should be warned that we deviate slightly from standard terminology in this sense. 

\subsubsection{}
The sense in which shift operators can be made ``nonabelian'' is the following. It is often the case that the torus $T$ acting on $X$ is a maximal torus in some larger reductive group $G$ which acts, and one may wish to work $G$-equivariantly rather than $T$-equivariantly. Since all counts in enumerative geometry factor through functors like $K_G(\text{---})$ or $H^\bullet_G(\text{---})$, at first sight one does not gain much since very often the only difference between $G$ and $T$-equivariance is taking invariants for the Weyl group $W$. 

Nonetheless, given a $G$-action on $X$ and a $G$-bundle on $C$, we are free to study the counts of sections of the $X$-bundle $\widetilde{X}$ associated to the $G$-bundle, and even versions of such counts as we allow the $G$-bundle to vary in moduli. In particular, given a point $0 \in C$ and a dominant cocharacter $\mu$ of $G$, we have a family of $G$-bundles on $C$ defined by the corresponding $G(\mathscr{O})$-orbit $G(\mathscr{O}) z^\mu$ inside the affine Grassmannian $\text{Gr}_G$ (we review these notions in more detail and give precise statements in Sections \ref{geometricdiffeq} and \ref{kthdiff}). We may study the counts of sections of the associated $\widetilde{X}$-bundles in this whole family, and denote the corresponding vertex functions by $\mathsf{Vertex}^\mu$. For the study of difference equations, it is sufficient to restrict to the case where $\mu$ is minuscule, which is a drastic simplification geometrically. 

We will refer to such counts of $\mu$-twisted maps, involving $X$-bundles associated to a $G$-bundle varying in moduli along an affine Grassmannian orbit, as nonabelian shift operators acting on vertex functions. 

\subsection{Statement of results}
\subsubsection{}
In a somewhat sketchy form, the main results of this paper are the following. First, in Section \ref{kthdiff} we study the counts of quasimaps to $X = T^*(GL_n/B)$ in equivariant $K$-theory. The vertex function $\mathsf{Vertex}(z, a)$ depends on variables $a \in A \subset GL_n$ valued in the maximal torus of $GL_n$, and variables $z \in A^\vee \subset GL_n^\vee$ valued in the Langlands dual torus. It depends on an additional variable $t \in \mathbb{C}^\times$ which is an equivariant variable associated to the scaling of the cotangent fibers of $T^*(GL_n/B)$, and a variable $q \in \mathbb{C}^\times$ which is an equivariant variable for the $\mathbb{C}^\times_q$ automorphism of the base curve mentioned above. 

A dominant minuscule cocharacter $\mu$ of $GL_n$ determines an orbit $G(\mathscr{O})z^\mu \simeq \text{Gr}(k, n)$ in the $GL_n$ affine Grassmannian and a representation $\Lambda^k(\mathbb{C}^n)$ of the dual group $GL_n^\vee \simeq GL_n$, which are mapped into each other via the geometric Satake equivalence \cite{ginzburg2000}, \cite{mirkovicvilonen}. We study the $\mu$-twisted quasimap counts, which count sections of an $X$-bundle associated to a $G$-bundle varying in moduli according to the orbit $G(\mathscr{O})z^\mu$, and show (Theorem \ref{heckeeigenval}) using geometric techniques developed in Section \ref{geometricdiffeq} that the count is in fact a multiple of the untwisted one
\begin{equation}
    \mathsf{Vertex}^\mu(z, a) = t^{\frac{1}{2} \dim \text{Gr}(k, n)}\chi_{\Lambda^k(\mathbb{C}^n)^*}(z_1 t^{\frac{n - 1}{2}}, \dots, z_nt^{\frac{1 - n}{2}}) \mathsf{Vertex}(z, a). 
\end{equation}
The explicit multiplier is the character of the representation of the dual group $GL_n^\vee$. On the other hand by a standard equivariant localization argument (Proposition \ref{vertexcomparison}), the twisted count $\mathsf{Vertex}^\mu(z, a)$ is related to the untwisted count by an equation of the form
\begin{equation}
    \mathsf{Vertex}^\mu(z, a) = \mathsf{Vertex}(z, a) \cdot \widehat{U}^{(a)}_\mu
\end{equation}
where $\widehat{U}^{(a)}_\mu$ is some explicit difference operator in the $a$-variables determined by the cocharacter $\mu$. In this paper we use the convention throughout that such difference operators act on functions from the \textit{right}. These two results together amount to a representation-theoretic identification of the difference equation in the $a$-variables.

In Section \ref{multmaps}, we study in a completely parallel fashion the cohomological counts to $X = GL_n/B$ (\textit{not} the cotangent bundle, this is not a typo), and obtain a very explicit formula (Theorem \ref{vertexint}) characterizing the cohomological vertex function as an intertwiner between two representations of the equivariant Borel-Moore homology $H_*^{\mathbb{C}^\times_q \ltimes G(\mathscr{O})}(\text{Gr}_G)$ of the $G = GL_n$ affine Grassmannian. The relation of this convolution algebra to integrable systems was first studied long ago by Bezrukavnikov, Finkelberg, and Mirkovic \cite{bfm}, and Theorem \ref{vertexint} gives a novel perspective on this connection which geometrizes the quantum inverse scattering method for the open Toda lattice (see \cite{sklyanin2000} for a review). The precise connection to the quantum inverse scattering method is the content of Theorem \ref{yangianquotient}, which connects the results of this paper directly to the study of comultiplication on Coulomb branch algebras \cite{bfnslice}, \cite{finkelberg2017}, \cite{krylovperunov}. 

\subsubsection{}
Identification of the difference equations satisfied by vertex functions is usually a nontrivial task, so at first sight it may seem that this paper just relates one difficult problem to another. One surprising simplification in the nonabelian situation is that there is a new set of tools to identify difference equations. Namely, for very elementary reasons explained in Section \ref{geometricdiffeq}, basic building blocks in all the moduli spaces relevant for computing the action of nonabelian shift operators are certain vector bundles over the minuscule $G(\mathscr{O})$-orbits (which collapse to points in the abelian case). 

These vector bundles come in families with Kahler moduli, and may undergo interesting flop transitions. One may study the wall-crossing behavior of various equivariant integrals over them. As the geometry of the minuscule Grassmannians is elementary, this is easy and turns out to give a strategy to deduce interesting identities among nonabelian shift operators. This technique is particularly powerful when $X$ is a quiver variety of type $A$, and we make heavy use of it in this paper to fully characterize the difference equations. The reader may notice that, for the study of complete flag varieties, we only need very basic special cases of the results we prove in Section \ref{geometricdiffeq}, and the inductive structure of the quiver. 

\subsection{Relation to Eisenstein series}
When the target space $X = G/B$ is a flag variety, a striking analogy emerges between enumerative geometry of curves in $X$ and the study of automorphic forms known as Eisenstein series \cite{laumon}, \cite{bravermangaits}, \cite{ko22}, \cite{ko24}. Nonabelian shift operators are closely related to this picture, in a sense we will now explain briefly and informally.

\subsubsection{}
If $C$ is now a more general algebraic curve defined over a finite field $\mathbb{F}_q$, automorphic forms are certain functions on the $\mathbb{F}_q$-points of the stack of bundles $\text{Bun}_G(C)$. For a variety $X$ with a $G$-action, a source of interesting functions are the counts of $\mathbb{F}_q$-points in moduli spaces of sections of $X$-bundles over $C$ associated to a given $G$-bundle. When $X = G/B$ such counts are automorphic forms referred to as Eisenstein series. If sections are replaced by quasisections, corresponding to compactifying maps by quasimaps, the counts are called geometric Eisenstein series (to be somewhat pedantic relative to the level of discussion here, in this paper we take $G = GL_n$ always and the notion of quasimaps to $G/B$ viewed as a quiver variety \cite{cfkimmaulik} corresponds to the quasimaps introduced by Laumon \cite{laumon}). 

\subsubsection{}
One is typically not interested in studying general functions on $\text{Bun}_G(C)$, but rather eigenfunctions of certain operators referred to as Hecke operators. Hecke operators take as input a point $p \in C$ and a cocharacter $\mu$ of $G$ that we may as well assume to be minuscule since $G = GL_n$. Fixing a bundle $\mathscr{E} \in \text{Bun}_G(C)$, this data determines a family of bundles via a $G(\mathscr{O})$-orbit in $\text{Gr}_G$ as above, which by construction are isomorphic to $\mathscr{E}$ upon restriction to $C \setminus p$. Hecke operators are defined by summing 
\begin{equation}
\mathscr{H}^\mu_pf(\mathscr{E}) = \sum_{\mathscr{E}'}f(\mathscr{E}')
\end{equation}
over all $\mathbb{F}_q$-points $\mathscr{E}'$ in the family. Hecke operators at different points commute, and Hecke eigenforms are functions which are eigenfunctions of all Hecke operators simultaneously. 

Eisenstein series are indeed Hecke eigenforms and the Langlands correspondence indexes such eigenforms by Galois representations, that is, local systems on $C$ for the Langlands dual group $G^\vee$. The local systems dual to geometric Eisenstein series are those with monodromy contained in a maximal torus $T^\vee \subset G^\vee$ \cite{laumon}, \cite{bravermangaits}. 

\subsubsection{}
At this point it must be clear that, essentially by definition, the nonabelian shift operators are precise analogs of Hecke operators in the geometric situation over $\mathbb{C}$, when the automorphic form is replaced by the vertex function. The results of Section \ref{kthdiff} provide an additional data point for the analogy of \cite{ko22}, \cite{ko24} and show that the $K$-theoretic vertex functions of $T^*(GL_n/B)$ are indeed eigenfunctions of nonabelian shift operators, with eigenvalues that are characters of $G^\vee$. In this scenario, we have returned to $C \simeq \mathbb{P}^1$ with equivariance, and in fact the result is purely local on the base curve.

We should note that, upon spelling out localization formulas, such a characterization of the difference equations for $\mathsf{Vertex}$ is essentially equivalent to that obtained earlier from the point of view of integrable systems in \cite{koroteevzeitlin} using explicit Mellin-Barnes type integral formulas for the vertex function, see also \cite{Koroteev_2021}, \cite{koroteev_2015}, \cite{Gaiotto_Koroteev}. Our approach is geometric and does not depend on integral formulas (though it is very closely related to them), so we hope that it will help build a bridge between integrability and the study of geometric automorphic forms. 

\subsubsection{}
It should be noted that categorification of the results in Section \ref{kthdiff} would appear to involve \textit{coherent} sheaves on $\text{Bun}_G(C)$, rather than the constructible sheaves which are usually used in the categorification of automorphic forms via the functions/sheaves correspondence. Ultimately, this is due to the fact that if one regards the enumerative computations we do here as certain computations in four-dimensional supersymmetric $G$-gauge theory along the lines of \cite{tam24} (see Appendix \ref{phys} for a lightning overview in the 3d context), they live in the holomorphic-topological twist \cite{kapustinht} of the 4d $\mathscr{N} = 2^*$ theory. The category of boundary conditions of this theory is very different in nature than that of the $A$-twisted topological gauge theory of \cite{kw} that is typically used to understand the automorphic side of geometric Langlands. 

\subsection{Application to Coulomb branches}
As another application of the techniques we develop here, we are able to gain insight into the representation theory of quantized Coulomb branch algebras. In this section we will provide some orientation and background to explain how our results fit into the (somewhat vast) existing literature on related topics. 

\subsubsection{}
Fix a quiver $Q$ of finite ADE type. If we frame the quiver and decorate the nodes and framing boxes with positive integers, it defines
a three-dimensional gauge theory with $\mathscr{N} = 4$ supersymmetry and the authors of \cite{bfn} give a mathematical definition of the Coulomb branch $\mathscr{M}_C$ of its moduli space of vacua. It is a holomorphic symplectic variety, and one may consider quantization $\widehat{\mathscr{M}}_C$ of its algebra of global functions (which is also described elegantly in the formalism of \cite{bfn}).

The algebra $\widehat{\mathscr{M}}_C$ is known to be a quotient of a quantum group called a shifted Yangian $\mathsf{Y}_{-\mu}(\mathfrak{g}_Q)$ based on the Lie algebra associated to the quiver $Q$ (the shift $\mu$ is a coweight determined by the node and framing
decoration of $Q$), see in particular Appendix B of \cite{bfnslice}. Proofs of this statement usually depend on presenting $\mathsf{Y}_{-\mu}(\mathfrak{g}_Q)$ by explicit generators and relations.

The results and general approach of this paper provide a different and complementary way to discover the Yangian symmetry underlying quantized Coulomb branch algebras $\widehat{\mathscr{M}}_C$. Yangians were introduced historically in the context of the quantum inverse scattering method, and it is in this fashion that they enter our analysis in Section \ref{multmaps}. In this way, the $RTT$ formalism \cite{frt} (also referred to as FRT formalism) for the quantum group emerges naturally in our enumerative computations. In fact, to our knowledge the $RTT$ formalism only exists at present for shifted
Yangians $\mathsf{Y}_{-\mu}(\mathfrak{g}_Q)$ in the case where $\mu$ is dominant \cite{Frassek_2022}; we plan to apply our geometric point of view to understand this issue in future work.

\subsubsection{}
The most basic reason for the relation of our work to Coulomb branches is that, essentially by construction, moduli spaces entering enumerative computations with nonabelian shift operators all have canonical maps to strata in the affine Grassmannian $\text{Gr}_G$ for the group $G$ acting on the target variety $X$. On the other hand, by the definition of Coulomb branches \cite{bfn}, elements of
the quantized Coulomb branch algebras are equivariant homology classes defined on spaces closely related to $\text{Gr}_G$. This opens up a pathway to relate the identification of the difference equations satisfied by $\mathsf{Vertex}$ to representation theory of quantized Coulomb branch algebras.

Our main result, Theorem \ref{vertexmult}, fully achieves this goal in the simplest family of examples when the relevant
Yangian $\mathsf{Y}_{-\mu}(\mathfrak{g}_Q)$ has $\mathfrak{g}_Q = \mathfrak{sl}_2$ and the shift $\mu$ is a multiple of the simple coroot $\alpha$.

\subsection{Outline of the paper}

\subsubsection{}
With the motivation and big picture explained, we can give a more detailed overview of what we actually do in this paper. In Section \ref{geometricdiffeq} we deduce wall-crossing formulas which are the technical basis of all difference equations we study. Section \ref{dilog} recalls basic facts about the simplest vertex functions, counting maps from $\mathbb{C}$ to a vector space. Section \ref{heckesection} gives an elementary discussion of the geometry behind nonabelian shift operators in this simple setting. Section \ref{diffeqflop} proves a wall-crossing formula for the Hirzebruch genera of moduli spaces that enter computations. Section \ref{satake} studies asymptotics of this wall-crossing formula from the point of view of the geometric Satake equivalence. 

Section \ref{kthdiff} revisits the difference equations for $K$-theoretic counts of quasimaps to $T^*(GL_n/B)$ in the sense of \cite{okounkovpcmi}. Section \ref{quiverquasi} recalls basic notions and definitions of the moduli spaces. Section \ref{vertexhecke} introduces the moduli spaces relevant for nonabelian shift operators and uses them to prove Theorem \ref{heckeeigenval}, stating that the vertex function is an eigenfunction of the nonabelian shift operators/Hecke operators. 

Section \ref{multmaps} descends to the cohomological setting to apply our techniques to the study of quantized Coulomb branch algebras \cite{bfn}, \cite{bfnslice}, in the simplest setting of the quantized Coulomb branch for pure gauge theory with $G = GL_n$. Section \ref{quantizedMc} defines and recalls basic features of the quantized Coulomb branch algebra $\widehat{\mathscr{M}}_C$. Section \ref{vertexmult} studies the interaction of the cohomological vertex function of $X = GL_n/B$ with nonabelian shift operators and deduces Theorem \ref{vertexint}. Section \ref{shiftedyang} elaborates on the sense in which Theorem \ref{vertexint} geometrizes the coproduct on the simplest shifted Yangian $\mathsf{Y}_{-n\alpha}(\mathfrak{sl}_2)$ in the $RTT$-formalism and gives a very short proof of Theorem \ref{yangianquotient}, that the quantized Coulomb branch algebra of pure gauge theory for $G = GL_n$ is indeed a quotient of $\mathsf{Y}_{-n \alpha}(\mathfrak{sl}_2)$ that we characterize fairly explicitly. 

\subsection{Future directions and relation to other work}

\subsubsection{Relation to prior work}
As is clear from the outline, in this paper we revisit enumerative geometry of flag varieties. This is itself a very well-studied subject and therefore there are many intersections of what we do here with previous work. In this section we acknowledge some of the work we believe has some relation to this paper in order to put the results in context.

The top left matrix component of our Theorem \ref{vertexint} recovers the celebrated characterization of the quantum cohomology and differential equation of $GL_n/B$ due to Givental-Kim \cite{Givental_1995} and Braverman \cite{braverman2004}, though our proof is totally different. See also the works of Koroteev and collaborators \cite{Koroteev_2021}, \cite{koroteev_2015} which discuss the $K$-theoretic vertex functions of $T^*(GL_n/B)$. 

The remarkable features of nonabelian shift operators we introduce in the context of quasimaps in this paper should be considered as an incarnation of a very general vision of Teleman \cite{teleman2014gauge}, see also the recent work of Gonzalez-Mak-Pomerleano \cite{gonzalezmakpomerleano}. Our nonabelian shift operators can be viewed as providing a concrete algebra-geometric setting for related computations in the context of quasimaps to quiver varieties, while the papers just referenced adopt a very different point of view based on symplectic topology. 

Also relevant is work of Hilburn and collaborators \cite{gammagehilburn}, \cite{gammagehilburnmazelgee} which provide mathematical foundations for the 3d $A$-model of topological field theory and (2-)categorical 3d mirror symmetry. The computations we perform in this paper may be viewed as particular computations of partition functions in the 3d $A$-model with boundary conditions and $\Omega$-deformation, see Appendix \ref{phys} and \cite{tam24} for details. As is typical in quantum field theory, it is much easier to perform computations than to give definitions. 

In the physics literature, the ideas we pursue are closely related to early works of Gaiotto, Koroteev and collaborators \cite{Cecotti_2014}, \cite{Gaiotto_Koroteev}, \cite{koroteev_2015}. See also the more recent paper by Ferrari and Zhang \cite{ferrarizhang} relating this setup to Berry connections. We should also point to the work of Bullimore, Dimofte, Gaiotto and Hilburn on boundary conditions in 3d $\mathscr{N} = 4$ gauge theories \cite{Bullimore_2016}. The conceptual origin for the Yangian symmetry of quantized Coulomb branch algebras is a certain universal setup involving D-branes in type IIA string theory considered by Costello, Gaiotto and Yagi \cite{Costello_2020}, \cite{costello2021qoperatorsthooftlines} and independently in unpublished work by Aganagic and Nekrasov \cite{aganagicnekrasov}; our computations fit into this framework and amount to detailed microscopic calculations using the effective theories living on defects. 

\subsubsection{Generalization to other Coulomb branches}
The main takeaway of Theorem \ref{vertexint} and its proof given in this paper is that the correct level of generality in which to study these structures is general with respect to the \textit{Coulomb branches}, rather than the target $X$ of the enumerative problem. That is to say, it is expected that an analog of Theorem \ref{vertexint} will hold for more general Coulomb branches $\mathscr{M}_C$ of quiver gauge theories (see \cite{bfn}, \cite{bfnslice} for definitions), but with the target $X$ always being a particular equivariant vector bundle over some flag variety. 

At the time of writing, we have proven the analog of Theorem \ref{vertexint} for all shifted Yangians of the form $\mathsf{Y}_{-\mu}(\mathfrak{sl}_2)$ with $\mu$ a dominant coweight, that is, all $A_1$ type quivers with rank $n$ gauge node and rank $m$ framing node with $m \leq 2n$. The $m = 2n$ case of this appears to be related to work of Lee and Nekrasov \cite{Lee_2021}. The details will appear in forthcoming joint work with S. Nair \cite{nairtam}. 

The structures we have found in general in \cite{nairtam} connect our computations with the theory of comultiplication on Coulomb branches and multiplication morphisms for generalized affine Grassmannian slices introduced in \cite{bfnslice} and studied in \cite{finkelberg2017}, \cite{krylovperunov}. We learned that a general conjecture of J. Hilburn implies that the enumerative computations of the kind we perform with nonabelian shift operators should always be related, by 3d mirror symmetry, to the multiplication morphisms for generalized affine Grassmannian slices. We aim to elaborate on and prove a version of this for ADE quiver gauge theories in \cite{nairtam}.

As a side remark, we note from \cite{afo} that there are very explicit Mellin-Barnes integral formulas for the quasimap vertex functions that we study. A corollary of our results is that these integrals for our choices of $X$ become explicit kernels for the quantized multiplication morphisms for Coulomb branches, though our proof of this result does not depend on the integral formula. This gives a geometric explanation of some observations of Gerasimov, Karchdev, and Lebedev made at the dawn of the subject \cite{gerasimov2003}, see also \cite{sklyanin2000}.

The multiplication morphisms are a great source of insight into the algebraic and geometric structures of Coulomb branches. Careful readers will note that the tools we develop in Sections \ref{geometricdiffeq} and \ref{multmaps} apply equally well to the vertex functions of all partial flag varieties $X = GL_n/P$, and give a characterization of their quantum differential and difference equations via Coulomb branches, which we elaborate on in \cite{nairtam}.

\subsection{Acknowledgements}
I am grateful to Yixuan Li, Yegor Zenkevich and Peng Zhou for useful discussions and feedback during the preparation of this work. Part of this work was completed during a visit to Columbia University, and I am grateful to the Department of Mathematics there for its hospitality. Results obtained in this project were announced at the Informal Mathematical Physics Seminar at Columbia (January 2024) and Berkeley Informal String-Math Seminar \cite{tamseminar}. 

I am especially grateful to Mina Aganagic and Andrei Okounkov for their interest in this project, providing very insightful feedback and emphasizing the role of the Satake isomorphism, as well as providing comments on an earlier draft. Special thanks also go to Vasily Krylov and Sujay Nair for many patient explanations on the multiplication morphisms, and to Sujay again for collaboration on \cite{nairtam}. I am also grateful to Justin Hilburn for generously sharing his ideas on related topics and his unpublished general conjectures on the structures we study here. 

\section{Genera, flops and difference equations} \label{geometricdiffeq}
There are two key geometric ingredients entering the difference and differential equations studied in this paper. The first, as reviewed in the introduction, is the incorporation of Hecke modifications of a $G$-bundle on the base curve $C$ to take into account a nonabelian symmetry group of the problem. The second is the invariance of certain equivariant genera under flop transformations. Both of these are illustrated in the following simplest nontrivial example, described at length in this section.

\subsection{Quantum dilogarithms and spaces of maps} \label{dilog}
Fix the base curve $C \simeq \mathbb{P}^1$, and let $U = \mathbb{P}^1 \setminus \{ \infty \}$. The group $\mathbb{C}^\times_q = \text{Aut}(\mathbb{P}^1, 0, \infty)$ acts preserving $U$. 

The infinite-dimensional vector space 
\begin{equation}
\mathsf{Maps}(\mathbb{C}_q \to \mathbb{C}_y) := H^0(U, \mathscr{O}_C) 
\end{equation}
carries an action of $\mathbb{C}^\times_q$ and an additional torus $\mathbb{C}^\times_y$ scaling the fiber of $\mathscr{O}_C$ with weight $-1$. With respect to this grading, it splits into an infinite direct sum of finite-dimensional subspaces. Corresponding to this fact, the equivariant Euler character of the structure sheaf on $\mathsf{Maps}(\mathbb{C}_q \to \mathbb{C}_y)$ converges to a well-defined analytic function in some region of the equivariant parameters:
\begin{equation}
    \chi(\mathsf{Maps}(\mathbb{C}_q \to \mathbb{C}_y), \mathscr{O}_{\mathsf{Maps}}) = \frac{1}{\varphi_q(y)}
\end{equation}
where 
\begin{equation}
    \varphi_q(y) := \prod_{n = 0}^\infty (1 - q^n y)
\end{equation}
and the infinite product on the right hand side is convergent for $|q| < 1$. The special function $\varphi_q(y)$ is referred to as the quantum dilogarithm or reciprocal $q$-Gamma function.

\subsubsection{} \label{charmaps}
It is convenient to work not only with $\mathscr{O}_{\mathsf{Maps}}$, but with the sheaves $\Omega^i_{\mathsf{Maps}}$ of polynomial holomorphic differential forms of degree $i$ on $\mathsf{Maps}$. For a parameter $t \in \mathbb{C}$, we have the following equivariant Euler character
\begin{equation}
\chi\Big( \mathsf{Maps}(\mathbb{C}_q \to \mathbb{C}_y), \sum_i (-t)^i \Omega^i_{\mathsf{Maps}} \Big) = \frac{\varphi_q(ty)}{\varphi_q(y)}
\end{equation}
which is computed by similar considerations as above. The $t \to 0$ limit recovers computations with $\mathscr{O}_{\mathsf{Maps}}$. 

\subsubsection{}
From the perspective of this paper, the $q$-difference equation
\begin{equation} \label{babydiff}
    \frac{1}{\varphi_q(y)} = \frac{1}{1 - y} \times  \frac{1}{\varphi_q(qy)}
\end{equation}
arises as a consequence of the exact sequence of sheaves on $C$
\begin{equation}
    0 \to \mathscr{O}_C(-1) \to \mathscr{O}_C \to \mathscr{O}_0 \to 0
\end{equation}
which is the ideal sheaf sequence associated to the subscheme $0 \in C$. Taking the character of the space of sections on $U$ gives exactly \eqref{babydiff}; the first factor on the RHS is from $\mathscr{O}_0$ and the second is from $\mathscr{O}_C(-1)$. 

\subsection{Behavior of sections under Hecke modification} \label{heckesection}
To describe the generalization of the above observation to higher rank bundles, it is convenient to review some notation and terminology used to describe Hecke modifications of vector bundles.

\subsubsection{} \label{affinegr}
Let $\mathscr{K} :=\mathbb{C}(\!( z )\!)$ denote the field of formal Laurent series in the variable $z$, and $\mathscr{O} := \mathbb{C}[\![ z]\!]$ the ring of formal Taylor series. The affine Grassmannian of a connected reductive group $G$ is defined abstractly as the moduli space of pairs $(\mathscr{P}, \varphi)$ where $\mathscr{P}$ is a $G$-bundle on $C$ and $\varphi: \mathscr{P} \eval_{C \setminus 0} \xrightarrow{\sim } \mathscr{P}^0 \eval_{C \setminus 0}$ is an isomorphism with the trivial bundle (in other words, a trivialization) away from $0 \in C$. Its set of $\mathbb{C}$-points may be presented as the quotient
\begin{equation} \label{affinegrass}
\text{Gr}_G := G(\mathscr{K})/G(\mathscr{O}). 
\end{equation}
The group $G(\mathscr{O})$ acts on $\text{Gr}_G$ from the left, and it admits a stratification by $G(\mathscr{O})$ orbits:
\begin{equation}
\text{Gr}_G = \bigsqcup_{\mu \in \Lambda^+_{\text{cochar}}} G(\mathscr{O}) z^\mu G(\mathscr{O})/G(\mathscr{O})
\end{equation}
where $\Lambda^+_{\text{cochar}}$ denotes the dominant cocharacters of $G$ and $z$ denotes the local coordinate on the affine chart $U \subset C$ (which by completing at $0 \in U$ we may also think of as a formal coordinate at zero). We introduce the standard notation $\text{Gr}^\mu_G := G(\mathscr{O}) z^\mu G(\mathscr{O})/G(\mathscr{O})$. It is easy to see that all the $G(\mathscr{O})$-orbits are finite dimensional. Likewise, it is easy to see that if $\mu$ is minuscule, the orbit $\text{Gr}^\mu_G$ is closed and isomorphic to a partial flag variety.

A bundle parameterized by a point in the affine Grassmannian is said to be obtained from the trivial bundle by a Hecke modification at $0$. A bundle corresponding to a point in $\text{Gr}^\mu_G$ is said to be obtained from the trivial bundle by a Hecke modification of type $\mu$. 

\subsubsection{}
For the applications considered in this paper, we will be concerned only with the group $G = GL_n$ and Hecke modifications associated to minuscule cocharacters. Dominant minuscule cocharacters of $GL_n$ are of the form $\mu = (1, \dots, 1, 0, \dots, 0)$ (or $(0, \dots, 0, -1, \dots, -1)$) for some number $k \leq n$ of $1$'s (or $-1$'s), which can be viewed canonically as highest weights of the $\Lambda^k(\mathbb{C}^n)$ representations of the Langlands dual group $G^\vee \simeq GL_n$ (or of the corresponding dual representations). 

An important sign convention we use in this paper is the following. If $\mu = (1, \dots, 1, 0, \dots, 0)$, then a bundle $\mathscr{W}^\mu_n$ is related to the rank $n$ trivial bundle by a Hecke modification of type $\mu$ at $0 \in C$ provided there is an exact sequence 
\begin{equation}
    0 \to \mathscr{O}_C^{\oplus n} \to \mathscr{W}^\mu_n \to \mathscr{O}_0 \otimes \mathbb{C}^k \to 0
\end{equation}
of sheaves on $C$. In this case, $\mathscr{W}^\mu_n$ corresponds to a point in $\text{Gr}^\mu_G \simeq \text{Gr}(k, n)$, the Grassmannian of $k$-dimensional subspaces in an $n$-dimensional vector space. The $\mathbb{C}^k$ factor above should be viewed canonically as the fiber of the rank $k$ tautological bundle $\mathscr{E}$ over $\text{Gr}(k, n)$ at the point corresponding to $\mathscr{W}^\mu_n$. It will be important in this paper to also study such sequences in families over $\text{Gr}^\mu_G$. 

\subsubsection{} \label{defineX}
Let $\mathscr{W}_n$ denote a trivial rank $n$ vector bundle over $C$. Introduce the infinite-dimensional variety of pairs
\begin{equation*}
    \mathcal{X}_{\mu, n} := \Bigg\{ (\mathscr{W}^{(\mu)}_n, s) \Big| \text{ $\mathscr{W}^{(\mu)}_n \in \text{Gr}^\mu_{GL_n}$ and $s \in H^0(U, \mathscr{W}^{(\mu)}_n)$} \Bigg\}.    
\end{equation*}
The notation $\mathscr{W}^{(\mu)}_n \in \text{Gr}^\mu_{GL_n}$ means that $\mathscr{W}^{(\mu)}_n$ is obtained from $\mathscr{W}_n$ by a Hecke modification of type $\mu$ at $0 \in \mathbb{P}^1$; it is indeed in natural correspondence with a point in $\text{Gr}^\mu_{GL_n}$ by the discussion above. 

Tracing through definitions, we conclude the following. First, $\text{Gr}^{(1, \dots, 1, 0, \dots, 0)}_{GL_n} \simeq \text{Gr}(k, n)$, the Grassmannian of $k$-dimensional subspaces in $\mathbb{C}^n$. This $\mathbb{C}^n$ is canonically isomorphic to the fiber of $\mathscr{W}_n$ over zero. Moreover, sections $s \in H^0(U, \mathscr{W}^{(\mu)}_n)$ are naturally identified with meromorphic vector-valued functions of the form 
\begin{equation}
    s(z) = \frac{s_{-1}}{z} + s_0 + s_1 z + \dots 
\end{equation}
where $s_{-1} \in \mathscr{E} \eval_{\mathscr{W}^{(\mu)}_n}$, the fiber of the rank $k$ tautological vector bundle $\mathscr{E}$ over $\text{Gr}(k, n)$ at the point $\mathscr{W}^{(\mu)}_n$. 

There is therefore a map 
\begin{equation} \label{XtoGr}
    \mathcal{X}_{\mu, n} \to \text{Tot}(\mathscr{E} \to \text{Gr}(k, n))
\end{equation}
given by forgetting all data except the Hecke modification and the value of the section at zero, the fibers of which are isomorphic to $\mathsf{Maps}(\mathbb{C}_q \to \mathbb{C}^n)$. 

\subsubsection{}
The group $\mathbb{C}^\times_q \times \text{Aut}(\mathscr{W}_n) \simeq \mathbb{C}^\times_q \times GL_n$ acts naturally on $\mathcal{X}_{\mu, n}$, and we are interested in performing computations in equivariant $K$-theory $K_{\mathbb{C}^\times_q \times GL_n}(\mathcal{X}_{\mu, n})$. 

As a preparation for this, it is convenient to review basic facts about the equivariant geometry of Grassmannians $\text{Gr}(k, n)$. The rank $k$ tautological bundle $\mathscr{E}$ fits into a short exact sequence with the trivial rank $n$ bundle
\begin{equation}
    0 \to \mathscr{E} \to \underline{\mathbb{C}}^n \to Q \to 0
\end{equation}
which defines the quotient bundle $Q$ and we have $T\text{Gr}(k, n) \simeq \text{Hom}(\mathscr{E}, Q)$. Fixed points of a maximal torus $T \subset GL_n$ correspond to the coordinate subspaces, which are indexed by subsets $I \subset \{1, \dots, n \}$ with $|I| = k$. For $(y_1, \dots, y_n) \in \text{Spec} \, K_T(\text{pt})$, we have 
\begin{equation*}
    \mathscr{E} \eval_{I} = \sum_{i \in I} y_i^{-1} \in K_T(\text{pt}).
\end{equation*}

\subsubsection{}
We are now in a position to state the following 

\begin{prop} \label{shift1}
We have the equivariant Euler character 
\begin{equation}
    \chi\Big( \mathcal{X}_{\mu, n}, \sum_i (-t)^i \Omega^i_{\mathcal{X}_{\mu, n}} \Big) = \sum_{\substack{I \subset \{1, \dots, n\} \\ |I| = k}} \prod_{i = 1}^n \frac{\varphi_q(t q^{-\delta_{i \in I}} y_i)}{\varphi_q(q^{- \delta_{i \in I}} y_i)} \prod_{\substack{i \in I \\ j \notin I}} \frac{1 - t y_j/y_i}{1 - y_j/y_i}.
\end{equation}
\end{prop}

\begin{proof}
Use the map \eqref{XtoGr} to pushforward the computation to $\text{Gr}(k, n)$ and apply equivariant localization on the Grassmannian. Use the discussion in Section \ref{charmaps} to compute the contributions from the fiber directions, which comprise the $\varphi_q$-factors in the sum.
\end{proof}

Introducing the shift operator 
\begin{equation}
    \widehat{U}_{\Lambda^k(\mathbb{C}^n)} := \sum_{\substack{I \subset \{ 1, \dots, n \} \\ |I| = k}} \prod_{i \in I} q^{D_i} \prod_{\substack{i \in I \\ j \notin I}} \frac{1 - t y_j/y_i}{1 - y_j/y_i}
\end{equation}
where $q^{D_i} \cdot y_j = q^{\delta_{ij}} y_j$, a corollary of Proposition \ref{shift1} and the existence of the map \eqref{XtoGr} is 
\begin{equation}
    \chi(\mathsf{Maps}(\mathbb{C}_q \to \mathbb{C}^n)) \cdot \widehat{U}_{\Lambda^k(\mathbb{C}^n)} = \chi(\text{Tot}(\mathscr{E} \to \text{Gr}(k, n))) \chi(\mathsf{Maps}(\mathbb{C}_q \to \mathbb{C}^n)) 
\end{equation}
where we introduced the abbreviation $\chi(X) := \chi(X, \sum_i (-t)^i \Omega^i_X)$ and the difference operator acts from the right. This equation contains no more content than the elementary difference equation \eqref{babydiff}, which is the basic reason why one typically studies Hecke modifications taking place in a maximal torus in the usual enumerative theory of quasimaps. We note that in spelling out localization formulas, it is important to take into account that it follows intrinsically from the definition of $\text{Gr}^\mu_{GL_n}$ as a moduli space that $\mathbb{C}^\times_q$ scales the fiber of $\mathscr{E}$ with weight $+1$. 

The $\widehat{U}_{\Lambda^k(\mathbb{C}^n)}$ operators are sometimes referred to as Macdonald operators and are one of the most important examples of quantized ($K$-theoretic) Coulomb branch operators from \cite{bfn}, \cite{bfnslice} that we review in Section \ref{multmaps}. 

Our goal in the rest of this section will be to use this geometric point of view to characterize more interesting difference equations satisfied by products of ratios of $\varphi_q$-functions, which do not follow immediately from the elementary $q$-difference equation \eqref{babydiff}. 

\subsection{Difference equations from flops} \label{diffeqflop}
The expression 
\begin{equation}
    \chi(X) := \chi \Big(X, \sum_i (-t)^i \Omega^i_X \Big)
\end{equation}
is known as the Hirzebruch genus of $X$. In practice, $X$ may very well be presented as a GIT quotient for a particular choice of a stability condition, and one of the most basic questions in wall-crossing concerns the behavior of such genera under variation of stability condition. In this section we will produce explicit wall-crossing formulas for specific $X$ arising in the context of Hecke modifications, and show that these formulas may be reinterpreted as explicit difference equations for products of ratios of $\varphi_q$-functions. 

\subsubsection{}
A straightforward generalization of the discussion in Section \ref{defineX} is to incorporate yet another trivial vector bundle $\mathscr{V}_m$ on the base curve $C$ of some rank $m$. Define the corresponding variety of pairs, for a minuscule cocharacter $\mu$ of $GL_n$, as
\begin{equation} \label{Xmunm}
\mathcal{X}_{(\mu, n), m} := \Bigg\{ (\mathscr{W}^{(\mu)}_n, s) \Bigg| \text{$\mathscr{W}^{(\mu)}_n \in \text{Gr}^\mu_{GL_n}$ and $s \in H^0(U, \mathscr{H}om(\mathscr{V}_m, \mathscr{W}^{(\mu)}_n))$} \Bigg\}. 
\end{equation}
Essentially the same discussion as in Section \ref{defineX} produces a map 
\begin{equation}
    \mathcal{X}_{(\mu, n), m} \to \text{Tot}(\text{Hom}(\mathbb{C}^m, \mathscr{E}) \to \text{Gr}(k, n))
\end{equation}
where we identify $\mathbb{C}^m$ as the fiber of $\mathscr{V}_m$ over zero. 

There is a natural action on $\mathbb{C}^\times_q \times \text{Aut}(\mathscr{W}_n) \times \text{Aut}(\mathscr{V}_m) \simeq \mathbb{C}^\times_q \times GL_n \times GL_m$ on $\mathcal{X}_{(\mu, n), m}$ and we perform computations in equivariant $K$-theory with respect to these symmetries. Fix maximal tori $T_y \subset GL_n$ and $T_x \subset GL_m$. A corollary of the proof of Proposition \ref{shift1} is that
\begin{equation}
    \chi(\mathcal{X}_{(\mu, n), m}) = \sum_{\substack{I \subset \{1, \dots, n \} \\ |I| = k}} \prod_{i = 1}^n \prod_{\ell = 1}^m \frac{\varphi_q(tq^{-\delta_{i \in I}} y_i/x_\ell)}{\varphi_q(q^{-\delta_{i \in I}}y_i/x_\ell)} \prod_{\substack{i \in I \\ j \notin I}} \frac{1 - ty_j/y_i}{1 - y_j/y_i}. 
\end{equation}

Equivalently, we have the following trivial difference equation satisfied by the Hirzebruch genera:
\begin{equation}
    \chi(\mathsf{Maps}(\mathbb{C}_q \to \text{Hom}(\mathbb{C}^m, \mathbb{C}^n)) \cdot \widehat{U}^{(y)}_{\Lambda^k(\mathbb{C}^n)} = \chi(\text{Tot}(\text{Hom}(\mathbb{C}^m, \mathscr{E}) \to \text{Gr}(k, n))) \chi(\mathsf{Maps}(\mathbb{C}_q \to \text{Hom}(\mathbb{C}^m, \mathbb{C}^n))). 
\end{equation}
The difference equation is called trivial because it follows directly from \eqref{babydiff}, and the superscript on $\widehat{U}$ is to remind us that it acts on the $y$-variables.

\subsubsection{} \label{Xkn}
When $m = n$, a symmetric role is played by $\mathscr{V}_n$ and $\mathscr{W}_n$. The space 
\begin{equation}
    X(k, n) := \text{Tot}(\text{Hom}(\mathbb{C}^n, \mathscr{E}) \to \text{Gr}(k, n))
\end{equation}
admits the following description as a GIT quotient:
\begin{equation}
    X(k, n) = \{ (I, J) \in \text{Hom}(\mathbb{C}^n_x, K) \oplus \text{Hom}^{\text{st}}(K, \mathbb{C}^n_y) \}/GL(K)
\end{equation}
where $\text{Hom}^{\text{st}}$ denotes linear maps of the maximal rank, and the vector space $K$ has dimension $k \leq n$. The vector space $\mathbb{C}^n_x$ is identified with the fiber of $\mathscr{V}_n$ over zero, and likewise $\mathbb{C}^n_y$ is identified with the fiber of $\mathscr{W}_n$ over zero. Tracing through definitions, $\mathbb{C}^\times_q$ scales the $\text{Hom}(\mathbb{C}^n_x, K)$ directions on the prequotient with weight $+1$ and acts trivially on the other directions. 

Viewed from this perspective, the symmetry between $\mathscr{V}_n$ and $\mathscr{W}_n$ is broken by the stability condition, and may be restored by comparing to computations on the flop 
\begin{equation}
    X^\vee(k, n) := \{ (I, J) \in \text{Hom}^{\text{st}}(\mathbb{C}^n_x, K) \oplus \text{Hom}(K, \mathbb{C}^n_y) \}/GL(K).
\end{equation}
$X^\vee$ and $X$ are abstractly isomorphic as algebraic varieties simply by exchanging the roles of $I$ and $J$, but this isomorphism is not $\text{Aut}(\mathscr{V}_n) \times \text{Aut}(\mathscr{W}_n)$-equivariant. This fact will play a key role in all that follows. 

We have a tautological bundle $\mathscr{K}$ on both $X, X^\vee$ induced from the trivial bundle with fiber $K$ on the prequotient, and the following formula for the equivariant $K$-theory class of the tangent bundle:
\begin{equation} \label{TX}
    TX = \mathbb{C}^n_y \otimes \mathscr{K}^* + q \mathscr{K} \otimes (\mathbb{C}^n_x)^* - \mathscr{K} \otimes \mathscr{K}^*
\end{equation}
valid for both $X$ and $X^\vee$. 

Fixed points in both $X$ and $X^\vee$ under the maximal torus $\mathbb{C}^\times_q \times T_x \times T_y \subset \mathbb{C}^\times_q \times \text{Aut}(\mathscr{V}_n) \times \text{Aut}(\mathscr{W}_n)$ are indexed by subsets $I \subset \{1, \dots, n \}$ with $|I| = k$. On $X$, we have 
\begin{equation}
    \mathscr{K} \eval_I = \sum_{i \in I} y_i^{-1} \in K_{\mathbb{C}^\times_q \times T_x \times T_y}(\text{pt}) 
\end{equation}
while on $X^\vee$, the formula is
\begin{equation}
    \mathscr{K} \eval_I = \sum_{\ell \in I} (qx_\ell)^{-1} \in K_{\mathbb{C}^\times_q \times T_x \times T_y}(\text{pt}). 
\end{equation}

\subsubsection{}
Now we may state 

\begin{theorem} \label{flopinv}
The Hirzebruch genus of $X(k, n)$ is invariant under flop, 
\begin{equation}
    \chi\Big( X, \sum_i (-t)^i \Omega^i_X \Big) = \chi\Big( X^\vee, \sum_i (-t)^i \Omega^i_{X^\vee} \Big). 
\end{equation}
\end{theorem}
While neither $X$ nor $X^\vee$ are proper, the Hirzebruch genera are well-defined as rational functions on $\text{Spec} \, K_{\mathbb{C}^\times_q \times \text{Aut}(\mathscr{V}_n) \times \text{Aut}(\mathscr{W}_n)}(\text{pt})$. In fact, Theorem \ref{flopinv} is equivalent to the following identity of rational functions:
\begin{equation} \label{explicitflopinv}
\sum_{\substack{I \subset \{ 1, \dots, n\} \\ |I| = k}} \prod_{\substack{i \in I \\ j \notin I}} \frac{1 - t y_j/y_i}{1 - y_j/y_i} \prod_{\substack{i \in I \\ \ell = 1, \dots, n}} \frac{1 - t q^{-1} y_i/x_\ell}{1 - q^{-1}y_i/x_\ell} = \sum_{\substack{I \subset \{ 1, \dots, n\} \\ |I| = k}} \prod_{\substack{i \in I \\ j \notin I}} \frac{1 - t x_i/x_j}{1 - x_i/x_j} \prod_{\substack{\ell \in I \\ i = 1, \dots, n}} \frac{1 - t q^{-1} y_i/x_\ell}{1 - q^{-1}y_i/x_\ell}.
\end{equation}

To place this explicit result in a wider context, invariance of elliptic genera (of which the $\chi$-genera we study here are certain degeneration limits) under flops has been studied in great generality, see for example \cite{borisov2000}, \cite{wang2002} or \cite{liu2024} for a more recent reference, thus in principle Theorem \ref{flopinv} follows on rather abstract grounds. We have decided to give a self-contained proof of it here for the convenience of readers, and because in Section \ref{satake} we will use specific features of the geometry of $\text{Gr}_G$ to achieve much greater control over wall-crossing formulas than is available in general.

Readers may also wish to consult Appendix \ref{contourint} for a discussion on the relation of wall-crossing formulas of this type to contour integrals. That such a relation exists should be considered well-known because of the ability to write integral formulas for equivariant localization on GIT quotients, see \cite{Moore_2000} and the appendix of \cite{afo}. 

\begin{proof}
Denote by $F_{k, n}(q, x, y)$ the genus of $X(k, n)$. Because the Grassmannian $\text{Gr}(k, n)$ is proper, $F_{k, n}$ has poles only along the divisors $y_i = qx_\ell$ corresponding to the codimension one subtori in $T := \mathbb{C}^\times_q \times T_x \times T_y$ with nonproper fixed loci in $X(k, n)$. 

To analyze the residue, consider the subtorus $S_{i \ell} \subset T$ defined by setting $y_j = 1$ for $j \neq i$ and $x_m = q^{-1}$ for $m \neq \ell$. It is elementary to see, using the quotient description, that the fixed locus of $S_{i \ell}$ consists of two connected components
\begin{equation}
X(k, n)^{S_{i \ell}} = X(k, n - 1) \bigsqcup X(k - 1, n - 1)
\end{equation}
and only the second $\text{Fix}_{i \ell} := X(k - 1, n - 1)$ has a normal weight which approaches zero as $y_i \to qx_\ell$. Then by $S_{i \ell}$-equivariant localization, 
\begin{equation}
F_{k, n}(q, x, y) = \chi\Bigg( \text{Fix}_{i\ell}, \sum_i \frac{(-t)^i \Omega^i_X \eval_{\text{Fix}_{i \ell}}}{\Lambda^\bullet(N^*_{X/\text{Fix}_{i \ell}})} \Bigg) + (\text{regular as $y_i \to qx_\ell$}). 
\end{equation}
It is understood that the Euler characteristic on $\text{Fix}_{i \ell}$ is taken $T/S_{i\ell}$-equivariantly. From the quotient description, we also see that the tautological bundle restricts as
\begin{equation}
    \mathscr{K} \eval_{\text{Fix}_{i \ell}} = y_i^{-1} \mathbb{C} + \mathscr{K}_{\text{Fix}_{i \ell}}. 
\end{equation}
Then from \eqref{TX} we conclude that 
\begin{equation}
    N_{X/\text{Fix}_{i \ell}} = qx_\ell/y_i + y_i \mathbb{C}^{n - 1}_y + qy_i^{-1} (\mathbb{C}^{n - 1}_x)^* + (qx_\ell - y_i) \mathscr{K}_{\text{Fix}_{i \ell}}.
 \end{equation}
All but the last term are trivial bundles on $\text{Fix}_{i \ell}$, thus contribute overall prefactors to the $\text{Fix}_{i \ell}$ contribution in the localization formula. The last term vanishes as $y_i \to qx_\ell$, thus does not contribute to the residue, and we conclude the asymptotics 
\begin{equation}
    F_{k, n} \sim \Big( \frac{1 - t}{1 - y_i/q x_\ell} \Big) \prod_{j (\neq i)} \frac{1 - t y_j/y_i}{1 - y_j/y_i} \prod_{m (\neq \ell)} \frac{1 - tx_\ell/x_m}{1 - x_\ell/x_m} \chi\Big( \text{Fix}_{i \ell}, \sum_i (-t)^i \Omega^i_{\text{Fix}_{i \ell}} \Big) = (\text{prefactors}) \times F_{k - 1, n - 1}
\end{equation}
as $y_i \to q x_\ell$. An almost identical calculation on the flop gives, denoting $F^\vee_{k, n}$ as the genus of $X^\vee(k, n)$, 
\begin{equation}
F^\vee_{k, n} \sim (\text{same prefactors}) \times F^\vee_{k - 1, n - 1}
\end{equation}
as $y_i \to qx_\ell$. Now we proceed by induction on $n$ taking the statement of the theorem as the inductive hypothesis, which is evidently true when $n = 1$. For the induction step, we use the analysis of the residues to conclude that $F_{k, n} - F^\vee_{k, n}$ has no pole as $y_i \to q x_\ell$ for each $i$ and $\ell$, and therefore is a regular function on $T$. Moreover by the explicit localization formulas given by the left and right hand sides of \eqref{explicitflopinv}, this function is bounded at all infinities of $T$, therefore is a constant on $T$ and depends on the variable $t$ only. 

Then we may evaluate the difference on any point of any toric compactification of $T$; let us choose a point on a divisor at infinity corresponding to the chamber $|y_1| \ll \dots \ll |y_n| \ll |x_1| \ll \dots \ll |x_n|$. In this limit we obviously have 
\begin{equation}
    F_{k, n} \to F^\vee_{k, n} \to \text{Poincar\'{e} polynomial of $\text{Gr}(k, n)$}
\end{equation}
therefore $F_{k, n} - F^\vee_{k, n} = 0$ identically. 
\end{proof}

\subsubsection{} \label{diffopsflop}
Theorem \ref{flopinv} admits the following more explicit reformulation. Introduce the difference operators 
\begin{equation}
\begin{split}
    \widehat{U}^{(y)}_{\Lambda^k(\mathbb{C}^n)} & = \sum_{\substack{I \subset \{ 1, \dots, n \} \\ |I| = k}} \prod_{i \in I} q^{D^{(y)}_i} \prod_{\substack{i \in I \\ j \notin I}} \frac{1 - ty_j/y_i}{1 - y_j/y_i} \\
    \widehat{U}^{(x)}_{\Lambda^k(\mathbb{C}^n)^*} & = \sum_{\substack{I \subset \{ 1, \dots, n \} \\ |I| = k}} \prod_{\substack{i \in I \\ j \notin I}} \frac{1 - tx_i/x_j}{1 - x_i/x_j} \prod_{i \in I} q^{D^{(x)}_i}
\end{split}
\end{equation}
where $q^{D^{(y)}_i} \cdot (x_\ell, y_j) = (x_\ell, q^{\delta_{ij}} y_j)$ and vice versa for $D^{(x)}_i$. Then we have 
\begin{equation}
    \chi(\mathsf{Maps}(\mathbb{C}_q \to \text{Hom}(\mathbb{C}^n_x, \mathbb{C}^n_y)) \cdot \widehat{U}^{(y)}_{\Lambda^k(\mathbb{C}^n)} = \widehat{U}^{(x)}_{\Lambda^k(\mathbb{C}^n)^*} \cdot \chi(\mathsf{Maps}(\mathbb{C}_q \to \text{Hom}(\mathbb{C}^n_x, \mathbb{C}^n_y))
\end{equation}
which is a rather nontrivial difference equation satisfied by 
\begin{equation}
    \chi(\mathsf{Maps}(\mathbb{C}_q \to \text{Hom}(\mathbb{C}^n_x, \mathbb{C}^n_y)) = \prod_{i = 1}^n \prod_{\ell = 1}^n \frac{\varphi_q(ty_i/x_\ell)}{\varphi_q(y_i/x_\ell)}. 
\end{equation}

\subsection{Asymptotics of localization formulas and geometric Satake equivalence} \label{satake}
The analysis of Section \ref{diffeqflop} deals with the situation where the vector bundles $\mathscr{V}_m$ and $\mathscr{W}_n$ on $C$ have the same rank, in other words $m = n$. In this section we study the case $m \neq n$, and deduce the corresponding difference equation. It is in fact a corollary of Theorem \ref{flopinv}, and the following elementary statement about asymptotics in the localization formula for the genus $\chi(X)$. 

\subsubsection{}
Let a torus $A$ act on a smooth variety $X$ and consider the equivariant Hirzebruch genus $\chi(X)$. This is a rational function of $a \in A$, and we are interested in its limit as $a \to 0$, where it is understood that this limit is taken in some chamber if $\text{rank} \, A > 1$. Let $F \in X^A$ denote a connected component of the fixed locus, then we have 

\begin{prop} \label{asymptotics}
\begin{equation}
    \lim_{a \to 0} \chi(X) = \sum_F t^{\text{rk} \, (N^{> 0}_{X/F})} \chi(F)
\end{equation}
where $N^{> 0}_{X/F}$ denotes the attracting part of the normal bundle with respect to the chamber $a \to 0$, and the sum runs over connected components of $X^A$. 
\end{prop}

\begin{proof}
By $A$-equivariant localization, the contribution of the normal bundle to $F$ to $\chi(X)$ is of the form 
\begin{equation}
\prod_{w} \frac{1 - tw^{-1}}{1 - w^{-1}}
\end{equation}
where $w$ runs over the equivariant Chern roots of $N_{X/F}$. As $a \to 0$, each factor approaches $1$ if $w$ is repelling and $t$ if $w$ is attracting.
\end{proof}

We remark that the proposition remains true if $A \subset T$ for some larger torus $T$ acting on $X$, if it is understood that $\chi(F)$ is taken $T/A$-equivariantly. If $X$ is proper, $\chi(X)$ does not depend on equivariant variables at all, and the above proposition reduces to the computation of the Poincar\'{e} polynomial by Morse theory. 

\subsubsection{} \label{chiGr}
One of the most basic illustrations of the principle discussed in the previous section is the computation of $\chi(\text{Gr}(k, n))$, which we briefly recall as it will be a useful intermediate step in the proof of Theorem \ref{Xknmflop}. 

For the obvious action of a maximal torus $T \subset GL_n$ on $\text{Gr}(k, n)$, the fixed locus consists of coordinate $k$-planes indexed by subsets $I \subset \{1, \dots, n\}$ of size $k$, with normal weights $y_i/y_j$ for each $i \in I$ and $j \notin I$. In the chamber $|y_1| \ll |y_2| \ll \dots \ll |y_n|$, for a fixed point $I = \{ i_1, \dots, i_k \}$, $i_1 < \dots < i_k$, the number of attracting directions is 
\begin{equation}
    k(n - k) - \sum_{j = 1}^k(i_j - j) = \frac{k(n - k)}{2} + \frac{k(n + 1)}{2} - \sum_{j = 1}^k i_j.
\end{equation}
An insightful way to organize the computation is to observe that, subtracting $\frac{1}{2} \dim \text{Gr}(k, n)$ from the above, what remains is the eigenvalue of the diagonal matrix 
\begin{equation}
    \text{diag}\Big( \frac{n - 1}{2}, \dots, \frac{1 - n}{2} \Big) 
\end{equation}
acting on the vector $e_{i_1} \wedge \dots \wedge e_{i_k} \in \Lambda^{k}(\mathbb{C}^n)$, where $e_i$, $i = 1, \dots, n$ denotes the standard basis of $\mathbb{C}^n$. Call this matrix $\rho_n$, as it is the image of the Weyl vector of $\mathfrak{sl}_2$ under the principal embedding $\mathfrak{sl}_2 \to \mathfrak{gl}_n$. We conclude from Proposition \ref{asymptotics} and properness of $\text{Gr}(k, n)$
\begin{equation}
    \chi(\text{Gr}(k, n)) = \lim_{y \to 0} \chi(\text{Gr}(k, n)) = t^{\frac{1}{2} \dim \text{Gr}(k, n)} \chi_{\Lambda^{k}(\mathbb{C}^n)}(t^{\rho_n}) 
\end{equation}
which is well-known. 

\subsubsection{}
Another, much more abstract, perspective on $\chi(\text{Gr}(k, n))$ is the following. One of the central structural results about the affine Grassmannian $\text{Gr}_G$ introduced in Section \ref{affinegr} is the geometric Satake equivalence \cite{ginzburg2000}, \cite{mirkovicvilonen} which asserts an equivalence of tensor categories 
\begin{equation}
    \text{Perv}_{G(\mathscr{O})}(\text{Gr}_G) \simeq \text{Rep}(G^\vee).
\end{equation}
On the left is the category of $G(\mathscr{O})$-equivariant perverse sheaves on the affine Grassmannian, with tensor structure defined by convolution. On the right is the category of finite-dimensional representations of the Langlands dual group, with the tensor structure given by the tensor product. 

To a dominant weight $\mu^\vee$ of $G^\vee$, there is an associated dominant coweight $\mu$ of $G$ and orbit closure $\overline{\text{Gr}}^\mu_G \subset \text{Gr}_G$. The geometric Satake equivalence maps the intersection cohomology complex $IC_\mu$ of the orbit closure to the irreducible representation of $G^\vee$ with highest weight $\mu^\vee$. 

If $\mu$ is minuscule, $\overline{\text{Gr}}^\mu_G = \text{Gr}^\mu_G$ is smooth and proper, and the intersection cohomology reduces to the ordinary de Rham cohomology of the orbit. For $G = GL_n$, the minuscule orbits are all of the form $\text{Gr}(k, n)$ for some $k$ and $n$. It is known that the grading of $H^*(\text{Gr}(k, n))$ by cohomological degree matches, up to a shift by $\frac{1}{2} \dim \text{Gr}(k, n)$, the grading of the corresponding representation $\Lambda^k(\mathbb{C}^n)$ of $G^\vee \simeq GL_n$ under the maximal torus of a principal $SL_2$ subgroup. This fact together with Hodge theory imply the result of the calculation of Section \ref{chiGr} on the level of characters of the principal $SL_2$. 

\subsubsection{}
Given a Levi subgroup $M \subset G$, there is a dual Levi $M^\vee \subset \, G^\vee$ and a natural restriction functor 
\begin{equation}
    \text{Rep}(G^\vee) \to \text{Rep}(M^\vee).
\end{equation}
An interesting question is that of the compatability of the geometric Satake equivalence with such restriction functors. The answer is known (\cite{beilinsondrinfeld}, Proposition 5.3.29) to be the following. Choose a parabolic $P$, which has maps $P \xhookrightarrow{} G$ and $P \to M$ given respectively by inclusion and quotient by unipotent radical. By covariance of the affine Grassmannian in $G$, there are associated maps 
\[\begin{tikzcd}
	& {\text{Gr}_P} \\
	{\text{Gr}_M} && {\text{Gr}_G}
	\arrow[from=1-2, to=2-1]
	\arrow[from=1-2, to=2-3]
\end{tikzcd}\]
and the restriction functor is dual to a certain pull-push in this diagram. 

\subsubsection{}
For our concrete applications, we just need to know what this diagram looks like for the minuscule orbits of $G = GL_n$, each of which is isomorphic to some Grassmannian $\text{Gr}(k, n)$. Levi subgroups $M \subset GL_n$ up to conjugacy are in one-to-one correspondence with partitions $n = n_1 + \dots + n_r$ of $n$. If a partition has $r$ parts, there is an associated rank $r$ torus $S \subset GL_n$, such that 
\begin{equation}
    \text{Gr}(k, n)^S = \bigsqcup_{k = k_1 + \dots + k_r} \prod_{i = 1}^r \text{Gr}(k_i, n_i). 
\end{equation}
The connected components of the fixed locus are precisely minuscule orbits in $\text{Gr}_M$. A choice of parabolic $P$ quotienting to $M$ is the same as a choice of chamber $\mathfrak{C}$ for the action of $S$, and the correspondence $\text{Gr}_P$ restricts to the correspondence 
\begin{equation}
    \text{Attr}_{\mathfrak{C}} \subset \text{Gr}(k, n)^S \times \text{Gr}(k, n)
\end{equation}
defined by attracting manifolds (in fact, $\text{Gr}_P$ itself may be viewed as an attracting manifold for a similarly defined torus action on $\text{Gr}_G$). Fixed loci of this form (for $r = 2$) and their attracting manifolds will play a key role in Theorem \ref{Xknmflop} below. 

\subsubsection{}
Let us return to our concrete situation, with vector bundles $\mathscr{V}_m$ and $\mathscr{W}_n$ on our curve $C \simeq \mathbb{P}^1$. We study Hecke modifications of $\mathscr{W}_n$ and $\mathscr{V}_m$ at $0 \in C$, and the behavior of sections under them. Without loss of generality we take $m < n$.

The natural analog of the space $X(k, n)$ defined in Section \ref{Xkn} in this scenario is 
\begin{equation}
    X(k, n, m) = \{ (I, J) \in \text{Hom}(\mathbb{C}^m_x, K) \oplus \text{Hom}^{\text{st}}(K, \mathbb{C}^n_y) \}/GL(K).
\end{equation}
We define $X^\vee(k, n, m)$ by placing the stability condition on $I$ instead of $J$. When $m \neq n$, $X$ and $X^\vee$ are no longer isomorphic. 

Fix a splitting $\mathbb{C}^n_x = \mathbb{C}^m_x \oplus \mathbb{C}^{n - m}_x$, and let a torus $\mathbb{C}^\times_\lambda$ act trivially on the first summand and scale the second summand with weight $-1$. There is an induced action of $\mathbb{C}^\times_\lambda$ of $X(k, n)$. It may be viewed as a subtorus of $T$, being the image of a cocharacter $\mathbb{C}^\times_\lambda \to T$ defined by the above splitting of $\mathbb{C}^n_x$. 

\begin{prop} \label{Xasymp}
The fixed locus 
\begin{equation}
    X(k, n)^{\mathbb{C}^\times_\lambda} = X(k, n, m)
\end{equation}
and the normal bundle is totally repelling in the chamber $\lambda \to \infty$.
\end{prop}

\begin{proof}
The normal bundle is, e.g. from \eqref{TX}
\begin{equation}
    N_{X/X^{\mathbb{C}^\times_\lambda}} = \lambda \mathscr{K} \otimes (\mathbb{C}^{n - m}_x)^*
\end{equation}
where $\mathscr{K}$ is the tautological bundle. Due to the overall factor of $\lambda$, it is totally repelling as $\lambda \to \infty$. 
\end{proof}

\subsubsection{}
As the action of $\mathbb{C}^\times_\lambda$ is induced by the splitting $\mathbb{C}^n_x = \mathbb{C}^m_x \oplus \mathbb{C}^{n - m}_x$, $\mathbb{C}^\times_\lambda$ also acts naturally on $X^\vee(k, n)$. We have the following result on the fixed locus and attracting directions.

\begin{prop} \label{Xflopasymp}
    The fixed locus 
    \begin{equation} \label{Xflopfixed}
        X^\vee(k, n)^{\mathbb{C}^\times_\lambda} = \bigsqcup_{k_1 + k_2 = k} X^\vee(k_1, n, m) \times \text{Gr}(k_2, n - m)
    \end{equation}
    where it is understood that we omit any component in which either of the two factors is empty, and we have 
    \begin{equation}
        \lim_{\lambda \to \infty} \chi(X^\vee(k, n)) = \sum_{k_1 + k_2 = k} t^{k_2(m - k_1) + \frac{1}{2}k_2(n - m - k_2)} \chi(X^\vee(k_1, n, m)) \chi_{\Lambda^{k_2}(\mathbb{C}^{n - m})}(t^{\rho_{n - m}}). 
    \end{equation}
\end{prop}

\begin{proof}
That the fixed locus takes the form \eqref{Xflopfixed} follows directly from an analysis using the quiver description. The factor $\text{Gr}(k_2, n - m)$ should be understood as $X^\vee(k_2, 0, n - m)$. Denote $\text{Fix}_{k_1, k_2} := X^\vee(k_1, n, m) \times \text{Gr}(k_2, n - m)$. We have for the restriction of the tautological bundle
\begin{equation}
    \mathscr{K} \eval_{\text{Fix}_{k_1, k_2}} = \mathscr{K}_1 + \lambda^{-1} \mathscr{K}_2
\end{equation}
where $\mathscr{K}_{1, 2}$ denote the pullbacks of the respective tautological bundles from either factor of $\text{Fix}_{k_1, k_2}$. From \eqref{TX}, it follows that 
\begin{equation}
    N_{X^\vee/ \text{Fix}_{k_1, k_2}} = \lambda^{-1}(\mathscr{K}_2 \otimes (\mathbb{C}^m_x)^* - \mathscr{K}_2 \otimes \mathscr{K}_1^*) + \lambda(\mathbb{C}^n_y \otimes \mathscr{K}_2^* + \mathscr{K}_1 \otimes (\mathbb{C}^{n - m}_x)^* - \mathscr{K}_1 \otimes \mathscr{K}_2^*) 
\end{equation}
thus $\text{rk} \, N^{> 0}_{X/\text{Fix}_{k_1, k_2}} = k_2(m - k_1)$. From Proposition \ref{asymptotics} we conclude 
\begin{equation}
    \lim_{\lambda \to \infty} \chi(X^\vee(k, n)) = \sum_{k_1 + k_2 = k} t^{k_2(m - k_1)} \chi(\text{Fix}_{k_1, k_2}).
\end{equation}
By multiplicativity of the $\chi$-genus and the results of Section \ref{chiGr}, we have 
\begin{equation}
    \chi(\text{Fix}_{k_1, k_2}) = \chi(X^\vee(k_1, n, m)) \times \chi(\text{Gr}(k_2, n - m)) = \chi(X^\vee(k_1, n, m)) \times t^{\frac{1}{2}k_2(n - m - k_2)} \chi_{\Lambda^{k_2}(\mathbb{C}^{n - m})}(t^{\rho_{n - m}})
\end{equation}
whence the conclusion. 
\end{proof}

\subsubsection{}
We are now in a position to state 

\begin{theorem} \label{Xknmflop}
    We have the wall-crossing formula 
    \begin{equation}
        \chi(X(k, n, m)) = \sum_{k_1 + k_2 = k} t^{k_2(m - k_1) + \frac{1}{2}k_2(n - m - k_2)} \chi_{\Lambda^{k_2}(\mathbb{C}^{n - m})}(t^{\rho_{n - m}}) \chi(X^\vee(k_1, n, m)). 
    \end{equation}
\end{theorem}

\begin{proof}
Using action of the torus $\mathbb{C}^\times_\lambda$ on $X(k, n)$ and $X^\vee(k, n)$, apply Proposition \ref{asymptotics} to the left and right hand side of Theorem \ref{flopinv}, then apply propositions \ref{Xasymp} and \ref{Xflopasymp}. 
\end{proof}
Using the difference operators of Section \ref{diffopsflop}, Theorem \ref{Xknmflop} may be reformulated as the difference equation 
\begin{equation}
\begin{split}
    \sum_{k_1 + k_2 = k} t^{k_2(m - k_1) + \frac{1}{2}k_2(n - m - k_2)} \chi_{\Lambda^{k_2}(\mathbb{C}^{n - m})}(t^{\rho_{n - m}}) \widehat{U}^{(x)}_{\Lambda^{k_1}(\mathbb{C}^m)^*} \cdot \chi(\mathsf{Maps}(\mathbb{C}_q \to \text{Hom}(\mathbb{C}^m_x, \mathbb{C}^n_y)) = \\ \chi(\mathsf{Maps}(\mathbb{C}_q \to \text{Hom}(\mathbb{C}^m_x, \mathbb{C}^n_y)) \cdot \widehat{U}^{(y)}_{\Lambda^{k}(\mathbb{C}^n)}
\end{split}
\end{equation}
satisfied by 
\begin{equation}
    \chi(\mathsf{Maps}(\mathbb{C}_q \to \text{Hom}(\mathbb{C}^m_x, \mathbb{C}^n_y)) = \prod_{i = 1}^n \prod_{\ell = 1}^m \frac{\varphi_q(ty_i/x_\ell)}{\varphi_q(y_i/x_\ell)}. 
\end{equation}
It is hopefully evident to the reader that this difference equation reflects the decomposition of the $GL_n$ representation
\begin{equation}
    \Lambda^k(\mathbb{C}^n) = \bigoplus_{k = k_1 + k_2} \Lambda^{k_1}(\mathbb{C}^m) \otimes \Lambda^{k_2}(\mathbb{C}^{n - m})
\end{equation}
under restriction to the Levi $M = GL_m \times GL_{n - m}$, viewed from the other side of the geometric Satake equivalence. 

\section{$K$-theoretic computations: Hecke eigenvalue property} \label{kthdiff}

\subsection{Quivers and quasimaps} \label{quiverquasi}
Quasimaps from a curve $C$ to an algebraic variety $X$ provide a natural way to compactify the naive space of maps $C \to X$, at least when $X$ may be presented as a geometric invariant theory (GIT) quotient. Moduli spaces of stable quasimaps to GIT quotients were constructed in \cite{cfkimmaulik} and have a perfect obstruction theory whenever $X$ is the GIT quotient of an affine variety with at most local complete intersection singularities by the action of a reductive group. We will be interested in studying quasimaps to $X = GL_n/B$. In this section we will recall some basic notions to establish notations and conventions. 

\subsubsection{}
The flag variety $X = GL_n/B$ can be viewed as the moduli space of flags 
\begin{equation}
    \{ 0 \} \subset V_1 \subset V_2 \subset \dots \subset V_{n - 1} \subset W_n
\end{equation}
of subspaces of a fixed $n$-dimensional vector space $W_n$, where $\dim V_i = i$. 

If we view the $V_i$ as abstract vector spaces, such flags are the same data as injective linear maps $s_i : V_i \to V_{i + 1}$ for each $i$, considered up to the natural action of $GL(V_i)$ for each $i$. That is, we have 
\begin{equation} \label{quotientpres}
    X = \Bigg\{ s \in \bigoplus_{i = 1}^{n - 1} \text{Hom}^{\text{st}}(V_i, V_{i + 1}) \Bigg\} \Bigg/ \prod_{i = 1}^{n - 1} GL(V_i).
\end{equation}
In writing this equation, we understand that $V_n = W_n$ and that $\text{Hom}^{\text{st}}(V_i, V_{i + 1}) \subset \text{Hom}(V_i, V_{i + 1})$ denotes the open subset of injective linear maps. This locus is in fact precisely the same as the GIT stable locus for the action of $\prod_{i = 1}^{n - 1} GL(V_i)$ on $\bigoplus_i \text{Hom}(V_i, V_{i + 1})$ for a certain choice of ample linearized line bundle. 

\subsubsection{}
A quasimap $f(z)$ from a smooth rational curve $C \simeq \mathbb{P}^1$ to $X$ consists, by definition, of the following data. We have a collection of vector bundles $\mathscr{V}_i$ on $C$, $i = 1, \dots, n - 1$ and $\text{rk} \mathscr{V}_i = i$, together with a section 
\begin{equation}
    f \in H^0(C, \mathscr{H}om(\mathscr{V}_1, \mathscr{V}_2) \oplus \dots \oplus \mathscr{H}om(\mathscr{V}_{n - 1}, \mathscr{W}_n)). 
\end{equation}
The bundle $\mathscr{W}_n$ is simply the rank $n$ trivial vector bundle on $C$. A section $f$, and the quasimap it defines, is called stable if the maps $\mathscr{V}_i \to \mathscr{V}_{i + 1}$ are injective on all but finitely many fibers. Equivalently, $f$ is stable if it embeds the locally free sheaf $\mathscr{V}_i$ as a coherent subsheaf of $\mathscr{V}_{i + 1}$. A quasimap $f$ is called nonsingular if it embeds $\mathscr{V}_i$ as a vector subbundle of $\mathscr{V}_{i + 1}$; such an $f$ defines a genuine map $f: C \to X$. 

The moduli space of stable quasimaps to $X$ is then
\begin{equation} \label{qmapdef}
    \mathsf{QM}(X) = \{ \text{bundles $\mathscr{V}_i$ with a stable section $f$} \} / \text{isomorphism}
\end{equation}
where we mod out by isomorphisms which act by the identity on $C$ and the fixed trivial framing bundle $\mathscr{W}_n$. The group $GL_n = \text{Aut}(\mathscr{W}_n)$ acts naturally on $\mathsf{QM}(X)$, as does the group $\mathbb{C}^\times_q = \text{Aut}(C, 0, \infty)$ of automorphisms of $C$ fixing two points. 

\subsubsection{}
The value of a stable quasimap $f$ at a point $p \in C$ is well-defined as a $\prod_i GL(V_i)$-orbit in the prequotient in \eqref{quotientpres}. In other words, the evaluation map $\text{ev}_p: \mathsf{QM}(X) \to \mathscr{X}$ is well defined as a map to the ambient quotient stack $\mathscr{X}$ containing $X$ as an open substack. The evaluation map to $X$ for any fixed $p$ is well-defined only as a rational map.

Then, fixing $\infty \in C \simeq \mathbb{P}^1$, there is a maximal open subset 
\begin{equation}
    \mathsf{QM}^\circ(X) \subset \mathsf{QM}(X)
\end{equation}
of quasimaps for which $\text{ev}_\infty: \mathsf{QM}^\circ(X) \to X$ is well-defined. These are precisely the quasimaps which are nonsingular at $\infty$. We will be interested in performing $K$-theoretic computations on $\mathsf{QM}^\circ(X)$.  

\subsubsection{}
The moduli space $\mathsf{QM}^\circ(X)$ is, rather surprisingly in the context of enumerative problems, a smooth quasiprojective variety (this can be seen immediately using the identification of the fibers of $\text{ev}_\infty$ with handsaw quiver varieties, and the smoothness of the latter \cite{nakajima2011}). 

The bundles $\mathscr{V}_i$, $\mathscr{W}_n$ entering the quasimap data clearly sweep out tautological bundles on $\mathsf{QM}^\circ(X) \times C$, which by abuse of notation we call by the same name. The $K$-theory class of the tangent bundle to $\mathsf{QM}^\circ(X)$ is expressed via tautological bundles as 
\begin{equation} \label{TQMkclass}
    T\mathsf{QM}^\circ(X) = \bigoplus_{i = 1}^{n -1} \text{Ext}^\bullet(\mathscr{V}_i, \mathscr{V}_{i + 1}) - \bigoplus_{i = 1}^{n - 1} \text{Ext}^\bullet(\mathscr{V}_i, \mathscr{V}_i) \in K_{\mathbb{C}^\times_q \times GL_n}(\mathsf{QM}^\circ(X))
\end{equation}
where we work in the $\mathbb{C}^\times_q \times GL_n$-equivariant $K$-group. All cohomology is taken along $C$, $\text{Ext}^\bullet := \text{Ext}^0 - \text{Ext}^1$, and we understand that $\mathscr{V}_n = \mathscr{W}_n$ whenever it appears in this formula.

\subsubsection{A remark on $X$ vs. $T^*X$} \label{Xvscotangent}
In modern enumerative problems, one is less interested in classical target spaces like $X = GL_n/B$ and more interested in targets with a symmetric perfect obstruction theory like $T^*(GL_n/B)$. The moduli spaces of quasimaps to each target are closely related, but the cotangent directions are well-known to enter computations by modifying the obstruction theory. It is also natural and important to take counts equivariantly with respect to the $\mathbb{C}^\times$ action scaling the cotangent directions, whence counts to $T^*(GL_n/B)$ depend on an additional parameter $t$. 

$T^*(GL_n/B)$ is a Nakajima quiver variety, and so can in principle be approached using the technology of \cite{okounkovpcmi}. However, with other applications in mind (in particular, categorification of the counts and applications to the Langlands program, see \cite{ko24}) the following strategy is more technically convenient. Simply \textit{define} the moduli space $\mathsf{QM}^\circ(T^*X)$ to be the derived scheme $T^*[1]\mathsf{QM}^\circ(X)$. It is easy to see that the square root $\mathscr{K}_\text{vir}^{1/2}$ exists, and after tensoring the canonically defined virtual structure sheaf by this line bundle, the counts will agree with those defined in \cite{okounkovpcmi} up to a shift of the $z$-variables (this shift will resurface in Theorem \ref{heckeeigenval} below). After this shift has been performed, the counts to $T^*X$ are essentially the Hirzebruch genera of $\mathsf{QM}^\circ(X)$. 

To reiterate, for the purposes of this paper, the $\mathbb{C}^\times_t$-equivariant counts to $T^* X$ can be taken by definition to be the Hirzebruch genera of $\mathsf{QM}^\circ(X)$. This exploits the fact that $T^*X$ is globally a cotangent bundle, which is certainly not true for general Nakajima varieties. 

\subsection{Vertex function and Hecke operators} \label{vertexhecke}
The principal object of interest in enumerative computations is the so-called vertex function, the generating function of equivariant curve counts in genus zero. In this section we will recall the definition of the $K$-theoretic vertex function of $T^*X$ in a form convenient for our purposes and define the nonabelian shift operators. We will give a geometric identification of the difference equation in equivariant variables satisfied by the vertex function, by inductively Hecke modifying the bundles $\mathscr{V}_i$ entering the quasimap data. A corollary of our proof is that the vertex function is an eigenfunction of the Hecke operators acting on the bundle $\mathscr{W}_n$, in an appropriate sense. 

\subsubsection{}
$\mathsf{QM}^\circ(X)$ consists of infinitely many finite-dimensional connected components, indexed by the degree of the quasimap $f$. By definition, 
\begin{equation}
    \text{deg} \, f = (\text{deg}(\mathscr{V}_1), \text{deg}(\mathscr{V}_2), \dots, \text{deg}(\mathscr{V}_{n - 1})) \in \mathbb{Z}^{n - 1}. 
\end{equation}
We let $\mathsf{QM}^\circ_{\text{deg} \, f}(X)$ denote a component of fixed degree. 

Introduce an $n$-tuple of coordinates $(z_1, \dots, z_n) \in (\mathbb{C}^\times)^n$. The generating function of counts of quasimaps to $T^*X$ is by definition 
\begin{equation} \label{vertexinstpart}
    \mathbf{Z} = \sum_{\text{deg} \, f} z^{\text{deg} \, f} \chi(\mathsf{QM}^\circ_{\text{deg} \, f}(X)) = (\mathsf{QM}^\circ (X) \to \text{pt})_*\Big( z^{\text{deg} \, f} \sum_i (-t)^i \Omega^i_{\mathsf{QM}^\circ(X)} \Big) \in K_{\mathbb{C}^\times_q \times \mathbb{C}^\times_t \times GL_n}(\text{pt})_{\text{localized}} \otimes \mathbb{Q}[\![ z ]\!].
\end{equation}
By definition, 
\begin{equation}
    z^{\text{deg} \, f} = \prod_{i = 1}^n z_i^{-(\text{deg}(\mathscr{V}_i) - \text{deg}(\mathscr{V}_{i - 1}))}.
\end{equation}
We remark that, due to the stability condition, the sum runs only over $\deg \, f$ in the effective cone $\deg \mathscr{V}_1 \leq \deg \mathscr{V}_2 \leq \dots \leq \deg \mathscr{V}_{n - 1} \leq \text{deg} \mathscr{W}_n = 0$. In fact, the series converges in a finite region of the $z_i$-variables as long as $|q| < 1$, so it is indeed appropriate to view $z_i$ as $\mathbb{C}^\times$-valued rather than as a formal variable. The coefficient of a given power of $z$ in the series is a polynomial in $t$ with coefficients that are rational functions on $\text{Spec} \, K_{\mathbb{C}^\times_q \times GL_n}(\text{pt})$. We will denote the equivariant variables for a maximal torus $A \subset GL_n$ by $(a_1, \dots, a_n) \in A \subset GL_n$. 

As will become clear momentarily, $z$ should be viewed canonically as an element of the Langlands dual torus $A^\vee \subset GL_n^\vee \simeq GL_n$.

\subsubsection{}
We will study the behavior of these counts as the framing bundle $\mathscr{W}_n$ is modified. Fixing a dominant minuscule cocharacter $\mu$ of $GL_n$, we define 
\begin{equation}
    \mathsf{QM}^{\circ, \mu}(X) := \{ \text{bundles $\mathscr{V}_i$ with a stable section $f$, \newline $\text{\, Hecke modification $0 \to \mathscr{O}_C^{\oplus n} \to \mathscr{W}^\mu_n \to \mathscr{O}_0 \otimes \mathscr{E} \to 0$}$} \}^\circ/\sim. 
\end{equation} 
As usual, the symbol $^\circ$ means we restrict to the open subset of such data where singularities of the quasimap are disjoint from $\infty \in C$, and notations for Hecke modifications are as in Section \ref{geometricdiffeq}. To be completely explicit, $f$ in our moduli space is a section of 
\begin{equation}
    f \in H^0(C, \mathscr{H}om(\mathscr{V}_1, \mathscr{V}_2) \oplus \dots \oplus \mathscr{H}om(\mathscr{V}_{n - 1}, \mathscr{W}^\mu_n))
\end{equation}
and we mod out by isomorphisms that are required to be $1$ on $C$ and the framing bundle as before. In English, the moduli space is almost the same except instead of pinning down $\mathscr{W}_n = \mathscr{O}_C^{\oplus n}$, we give it a small degree of freedom to move between the first and third term in the exact sequence describing Hecke modification. 

Then the evaluation map $\text{ev}_\infty: \mathsf{QM}^{\circ, \mu}(X) \to X$ is well-defined, since by definition of Hecke modification we are given a trivialization $\mathscr{W}^\mu_n \eval_{C \setminus 0} \xrightarrow{\sim} \mathscr{O}_C^{\oplus n} \eval_{C \setminus 0}$ on the complement of $0 \in C$. For the same reason the group $GL_n$ continues to act on the moduli space by changing the trivialization away from $0 \in C$. There is a canonical $GL_n$-equivariant map to the corresponding stratum of the affine Grassmannian
\begin{equation}
    \mathsf{QM}^{\circ, \mu}(X) \to \text{Gr}^\mu_{GL_n}. 
\end{equation}
Note that $\text{Gr}^\mu_{GL_n}$ is smooth and proper when $\mu$ is minuscule. The fibers of this projection may be identified with spaces of twisted quasimaps, in the language of \cite{okounkovpcmi}. This map gives $\mathsf{QM}^{\circ, \mu}(X)$ a natural obstruction theory, with virtual tangent space given by the sum of the virtual tangent space along the fibers as discussed in \cite{okounkovpcmi} together with the pullback of the tangent bundle to $\text{Gr}^\mu_{GL_n}$ under the canonical map. This comment applies to more general targets $X$, and we may even contemplate modifying the obstructions by pulling back additional bundles from $\text{Gr}^\mu_{GL_n}$. The latter generalization is relevant to the study of Coulomb branches of gauge theories with matter \cite{bfn}, \cite{bfnslice}. 

Then in particular for $X = GL_n/B$ we have a well-defined twisted count
\begin{equation}
    \mathbf{Z}^\mu = \sum_{\text{deg} \, f} z^{\text{deg} \, f} \chi(\mathsf{QM}^{\circ, \mu}_{\text{deg} \, f}(X)). 
\end{equation}
Geometrically, this may be viewed as a virtual count of (quasi)sections of a nontrivial $T^* X$ bundle over $C$ determined by $\mathscr{W}_n^{\mu}$, where we additionally allow the bundle to vary along the $\text{Gr}^\mu_{GL_n}$ directions. All counts are performed with equivariance under the natural action of $\mathbb{C}^\times_q$ and $GL_n$, as usual. 

Our first goal is to relate the twisted count to the untwisted count. By  $\mathbb{C}^\times_q \times A$-equivariant localization, this can be reduced to the comparison already established in Lemma 8.2.12 of \cite{okounkovpcmi}. For the convenience of readers, we provide a self-contained argument in the next few sections. Much of what follows is true for rather general targets $X$, so we try to use specific features of flag varieties as minimally as possible. 

\subsubsection{}
First, since we now count sections of nontrivial bundles over $C$, the degree of a constant map may already be nontrivial, and this must be taken into account. 

To state this more precisely, let $\sigma: \mathbb{C}^\times \to GL_n$ be a cocharacter. To this we may associate the fixed locus $X^\sigma \subset X$ under the corresponding one-parameter subgroup. We may also construct a principal $GL_n$-bundle over $\mathbb{P}^1$ using $\sigma$ as a clutching function, and pass to the associated $X$-bundle. We denote this $X$-bundle by $\widetilde{X}$. 

Each element $x \in X^\sigma$ defines a ``constant'' section of $\widetilde{X}$. By definition, a section of $\widetilde{X}$ is the same thing as a flag of subbundles $\mathscr{V}_1 \subset \mathscr{V}_2 \subset \dots \subset \mathscr{W}_n^\sigma$, where $\mathscr{W}_n^\sigma$ is the rank $n$ vector bundle associated to the underlying principal $GL_n$-bundle. An element $x \in X^\sigma$ is precisely the same thing as a choice of $\mathbb{C}^\times$-module structure on each vector space $V_i$ in \eqref{quotientpres} (with $W_n$ regarded as a $\mathbb{C}^\times$-module via $\sigma$) and a choice of $\mathbb{C}^\times$-equivariant quiver maps $V_i \to V_{i + 1}$ (taken modulo the centralizer of $\mathbb{C}^\times$ in $\prod_i GL(V_i)$). Then, for each $x \in X^\sigma$, we have an associated flag of vector bundles $\mathscr{V}_i^x$ on $\mathbb{P}^1 = (\mathbb{C}^2 \setminus 0 )/\mathbb{C}^\times$, which is the ``constant'' section. Its degree is  
\begin{equation}
    \text{deg} \, (x \in X^\sigma) = (\text{deg} \mathscr{V}_1^x, \dots, \text{deg} \mathscr{V}_{n - 1}^x)
\end{equation}
which depends on $x$ in a locally constant fashion.

If the image of $\sigma$ lies inside some maximal torus $A \subset GL_n$, we have an inclusion $X^A \subset X^\sigma$ and we define an element $z^{\text{deg}(\text{const})} \in K_A(X)_{\text{localized}} \otimes \mathbb{Q}(\!( z )\!)$ by restrictions
\begin{equation}
    z^{\text{deg}(\text{const})} \eval_{p \in X^A} = \prod_{i = 1}^n z_i^{-(\text{deg}(\mathscr{V}_i^p) - \text{deg}(\mathscr{V}_{i - 1}^p))}.
\end{equation}

\subsubsection{}
For the statement of the next proposition, it is convenient to extend the $\varphi_q$-function introduced in Section \ref{dilog} to $K$-theory as a genus. This means that, if $\mathscr{E}$ and $\mathscr{F}$ are $K$-classes of vector bundles, for the virtual bundle $\mathscr{E} - \mathscr{F}$ we put
\begin{equation}
\varphi_q(\mathscr{E} - \mathscr{F}) = \prod_{x \in \text{Chern roots of $\mathscr{E}$}} \varphi_q(x) \times \prod_{y \in \text{Chern roots of $\mathscr{F}$}} \frac{1}{\varphi_q(y)}. 
\end{equation}
For each $|q| < 1$ this defines a certain analytic $K$-theory class, that is, a section of the sheaf of analytic functions on the spectrum of the $K$-theory ring. 

\subsubsection{}
For the statement of the next proposition, note we can factor the pushforward in \eqref{vertexinstpart} through evaluation at infinity: 
\begin{equation}
    \mathbf{Z} = (X \to \text{pt})_* (\text{ev}_{\infty})_*(z^{\text{deg} \, f} \mathscr{O}^{\text{vir}}_{\mathsf{QM}^\circ(X)}).
\end{equation}
Note also that the Weyl group $W$ of $GL_n$ acts naturally on the cocharacters, and the $A$-fixed locus inside $\text{Gr}^\mu_{GL_n}$ consists precisely of cocharacters $\sigma: \mathbb{C}^\times \to A$ in the Weyl orbit $W \cdot \mu$ of the dominant cocharacter $\mu$. Then we have the following 

\begin{prop} \label{instcomparison}
\begin{equation} \label{comparisonforZ}
\mathbf{Z}^\mu = \sum_{\sigma \in W \cdot \mu} \frac{\Lambda^\bullet(t T^*_\sigma \text{Gr}^\mu_{GL_n})}{\Lambda^\bullet(T^*_\sigma \text{Gr}^\mu_{GL_n})} \times (X \to \text{pt})_*\Bigg( z^{\text{deg}(\text{const})} \frac{\varphi_q(q(1 - t)T^*X)}{\varphi_q(q(1 - t)T^*X) \eval_{a \mapsto aq^{-\sigma}}} \times \Big( (\text{ev}_\infty)_*(z^{\text{deg} \, f} \mathscr{O}^{\text{vir}}_{\mathsf{QM}^{\circ}(X)}) \Big) \eval_{a \mapsto aq^{-\sigma}} \Bigg). 
\end{equation}
The pushforwards are defined by $\mathbb{C}^\times_q \times A$-equivariant localization, and the notation $a \mapsto aq^{-\sigma}$ means that in the localization formula we replace the $A$-weights evaluated on $a \in A$ with the $A$-weights evaluated on $\sigma(q)^{-1} \cdot a \in A$. 
\end{prop}

\begin{proof}
The proof is just an analysis of the equivariant localization formula for $\mathbf{Z}^\mu$, using the action of $\mathbb{C}^\times_q \times A$ on $\mathsf{QM}^{\circ, \mu}(X)$. The $A$-fixed points on $\text{Gr}^\mu_{GL_n}$ are the cocharacters $\sigma$ entering \eqref{comparisonforZ} above, and we may analyze the contributions of $\mathbb{C}^\times_q \times A$-fixed points in the fibers as follows. 

It is elementary to see that the fixed locus of $\mathbb{C}^\times_q \times A$ acting on $\mathsf{QM}^\circ(X)$ and $\mathsf{QM}^{\circ, \mu}(X)$ consists of isolated points, and that there is a degree-shifting bijection between $\mathbb{C}^\times_q \times A$-fixed points in $\mathsf{QM}^\circ(X)$ and $\mathbb{C}^\times_q \times A$-fixed points in the fiber of $\mathsf{QM}^{\circ, \mu}(X)$ over any $A$-fixed point $\sigma \in \text{Gr}^\mu_{GL_n}$. 

Let $\mathscr{T}_X = \bigoplus_i \mathscr{V}_i \otimes \mathscr{V}_{i - 1}^* - \bigoplus_i \mathscr{V}_i \otimes \mathscr{V}_i^* \in K_{\mathbb{C}^\times_q \times GL_n}(\mathsf{QM}^\circ(X) \times C)$. Note $\mathscr{T}_X \eval_\infty \simeq \text{ev}_\infty^*(TX)$. Then equivariant Grothendieck-Riemann-Roch applied to \eqref{TQMkclass} gives 
\begin{equation}
    T\mathsf{QM}^\circ(X) \eval_{\text{fixed point}} = \frac{\mathscr{T}_X \eval_0}{1 - q^{-1}} + \frac{\text{ev}_\infty^*(TX)}{1 - q}
\end{equation}
with it understood the virtual vector bundles in the numerators are restricted to the corresponding fixed point, that is, they are just Laurent polynomials on $\mathbb{C}^\times_q \times A$. By definition of Hecke modification, and due to the bijection discussed above, over any fixed point $\sigma \in \text{Gr}^\mu_{GL_n}$ we have 
\begin{equation}
    T\mathsf{QM}^{\circ, \mu}(X) \eval_{(\sigma, \text{fixed point})} = T_\sigma \text{Gr}^\mu_{GL_n} + \frac{\mathscr{T}_X \eval_0}{1 - q^{-1}} \eval_{a \mapsto a q^{-\sigma}} + \frac{\text{ev}_\infty^*(TX)}{1 - q}. 
\end{equation}
We conclude immediately that 
\begin{equation}
    T \mathsf{QM}^{\circ, \mu}(X) \eval_{(\sigma, \text{fixed point})} - T\mathsf{QM}^\circ(X) \eval_{\text{fixed point, $a \mapsto aq^{-\sigma}$}} = T_\sigma \text{Gr}^\mu_{GL_n} + \frac{\text{ev}_\infty^*(TX)}{1 - q} - \frac{\text{ev}_\infty^*(TX)}{1 - q} \eval_{a \mapsto a q^{-\sigma}}.
\end{equation}
Now everything but the first term is pulled back under $\text{ev}_\infty$, so pulls out of the pushforward $(\text{ev}_\infty)_*$. Finally observe that 
\begin{equation}
    \frac{\text{ev}_\infty^*(TX)}{1 - q} = - \frac{q^{-1} \text{ev}_\infty^*(TX)}{1 - q^{-1}}
\end{equation}
whence, if we want to write formulas in terms of infinite products convergent for $|q| < 1$, this contributes $\varphi_q(q(1 - t) T^*X)$ in localization formulas with the virtual structure sheaf as we have defined it. Taking into account the degree shift in the bijection of fixed points we conclude exactly \eqref{comparisonforZ}. 
\end{proof}

\subsubsection{}
The statement of \eqref{comparisonforZ} can be simplified considerably if we introduce the appropriately normalized \textit{vertex functions}, defined as 
\begin{equation} \label{normalizedvertex}
    \mathsf{Vertex}^\mu = (X \to \text{pt})_*\Bigg( z^{\frac{\log \det V}{\log q}} \times \frac{\varphi_q(qtT^* X)}{\varphi_q(qT^*X)} \times (\text{ev}_\infty)_*( z^{\text{deg} \, f} \mathscr{O}^{\text{vir}}_{\mathsf{QM}^{\circ, \mu}(X)}) \Bigg)
\end{equation}
We write $\mathsf{Vertex} := \mathsf{Vertex}^{\mu = 0}$ for the vertex function with trivial twist. We introduced the notation 
\begin{equation}
    z^{ \frac{\log \det V}{\log q}} := \prod_{i = 1}^n z_i^{-\frac{\log(\det V_i) - \log( \det V_{i - 1})}{\log q}} 
\end{equation}
where logarithms are defined by substitution of Chern roots. For each $z$ this gives rise to a localized equivariant $K$-theory class on $X$, see Section 8.2 of \cite{okounkovpcmi} for more discussion on logarithms of line bundles in equivariant $K$-theory. 

The vertex functions are analytic functions of both the $z$ and $a$-variables (although defined a priori as a formal power series in $z$). We will sometimes write $\mathsf{Vertex}(z, a)$ to emphasize this dependence. The variables $q$ and $t$ are usually regarded as fixed, with $|q| < 1$. 

The vertex functions introduced above differ more or less by normalizations with \eqref{vertexinstpart}. The prefactors are referred to as tree-level and perturbative contributions, in reference to the interpretation of $\mathsf{Vertex}$ as a partition function in a certain supersymmetric quantum field theory. In that language, the function $\mathbf{Z}$ contains only the instanton contributions to the vertex function.  

In any case, with normalizations taken into account Proposition \ref{instcomparison} may be rephrased as 
\begin{prop} \label{vertexcomparison}
Suppose we choose the minuscule dominant cocharacter $\mu$ to be the highest weight of the $\Lambda^k(\mathbb{C}^n)$ representation of $GL_n^\vee$. Then we have 
\begin{equation}
    \mathsf{Vertex}^\mu(z, a) = \mathsf{Vertex}(z, a) \cdot \widehat{U}^{(a)}_{\Lambda^k(\mathbb{C}^n)}
\end{equation}
with the difference operator $\widehat{U}^{(a)}_{\Lambda^k(\mathbb{C}^n)}$ as in Section \ref{diffopsflop}. 
\end{prop}

\subsubsection{}
Proposition \ref{vertexcomparison} gives a geometric interpretation of the difference operator $\widehat{U}_{\Lambda^k(\mathbb{C}^n)}$ as integration over the orbit $\text{Gr}^\mu_{GL_n} \subset \text{Gr}_{GL_n}$. Now we will use this interpretation to arrive at a difference equation for $\mathsf{Vertex}(z, a)$ in the $a$-variables, by using the results of Section \ref{geometricdiffeq}. The basic idea will be to use the results there to formulate a different comparison of the twisted count $\mathsf{Vertex}^\mu$ and untwisted count $\mathsf{Vertex}$. 

This is done as follows. By definition of Hecke modification, the modified framing bundle $\mathscr{W}_n^\mu$, viewed as a bundle on $\mathsf{QM}^{\circ, \mu}(X) \times C$, fits into the following exact sequence of sheaves
\begin{equation}
    0 \to \mathscr{W}_n \to \mathscr{W}^\mu_n \to \mathscr{O}_0 \otimes \mathscr{E} \to 0
\end{equation}
where $\mathscr{O}_0$ is the pullback of the structure sheaf of $0 \in C$ and the bundle $\mathscr{E}$ is the pullback of the tautological bundle on the moduli space of Hecke modifications $\text{Gr}^\mu_{GL_n} \simeq \text{Gr}(k, n)$ for some $k$ determined by $\mu$. Observe that, because we view $\text{Gr}(k, n)$ as the moduli space of Hecke modifications, $\mathbb{C}^\times_q$ scales the fibers of $\mathscr{E}$ by the character $q$. 

The associated long exact sequence of $\text{Ext}$-groups along $C$ reads, in $K$-theory of quasimap moduli,
\begin{equation}
    \text{Ext}^\bullet(\mathscr{V}_{n - 1}, \mathscr{W}^\mu_n) = \text{Ext}^\bullet(\mathscr{V}_{n - 1}, \mathscr{W}_n) + \text{Ext}^\bullet(\mathscr{V}_{n - 1}, \mathscr{O}_0 \otimes \mathscr{E}) = \text{Ext}^\bullet(\mathscr{V}_{n - 1}, \mathscr{W}_n) + \mathscr{V}^*_{n - 1} \eval_0 \otimes \mathscr{E}. 
\end{equation}

This allows us to relate the tangent bundle to $\mathsf{QM}^{\circ, \mu}(X)$ and $\mathsf{QM}^\circ(X)$, viewed as $K$-theory classes, in a fairly direct fashion (compare to the discussion in Section \ref{defineX}):
\begin{equation}
    T\mathsf{QM}^{\circ, \mu}(X) = T\mathsf{QM}^\circ(X) + T\text{Gr}^\mu_{GL_n} +  \mathscr{V}_{n - 1}^* \eval_0 \otimes \mathscr{E} \in K_{\text{eq}}(\mathsf{QM}^{\circ, \mu}(X)). 
\end{equation}

The equivariant Hirzebruch genera $\chi(X(k, n, m))$ of the spaces from Proposition \ref{Xasymp} define rational functions on $\text{Spec} \, K_{\mathbb{C}^\times_q \times GL_n \times GL_m}(\text{pt})$, whence localized $K$-theory classes on quasimap moduli for each pair of vector bundles of rank $n$ and $m$ by formal substitution of Chern roots. For $m = n - 1$, we denote such a class by $\chi(k, \mathscr{W}_n \eval_0, \mathscr{V}_{n - 1} \eval_0)$.

Finally, introduce the notation $\langle \mathscr{F} \rangle := (X \to \text{pt})_*(\text{prefactors} \times (\text{ev}_{\infty})_*(z^{\text{deg} f} \mathscr{O}^{\text{vir}}_{\mathsf{QM}^\circ(X)} \otimes \mathscr{F}))$ where $\mathscr{F}$ is any (perhaps localized) equivariant $K$-theory class on $\mathsf{QM}^\circ(X)$ and the prefactors are as in \eqref{normalizedvertex}. 

The discussion above amounts to the following 
\begin{prop} \label{heckeinsertion}
Let $\mu$ and $k$ be as in Proposition \ref{vertexcomparison}. Then we have
\begin{equation}
    \mathsf{Vertex}^\mu = \Big\langle z_n^{- \deg(\mathscr{W}^\mu_n)} \chi \Big(k, \mathscr{W}_n \eval_0, \mathscr{V}_{n - 1} \eval_0 \Big) \Big\rangle. 
\end{equation}
\end{prop}

\subsubsection{}
Now we are in a position to state and prove

\begin{theorem} \label{heckeeigenval}
The normalized vertex function $\mathsf{Vertex}(z, a)$ satisfies the difference equation in $a$-variables
\begin{equation}
    \mathsf{Vertex} \cdot \widehat{U}^{(a)}_{\Lambda^k(\mathbb{C}^n)} = t^{\frac{1}{2} \dim \text{Gr}(k, n)} \chi_{\Lambda^k(\mathbb{C}^n)^*}(z_1 t^{\frac{n - 1}{2}}, \dots, z_n t^{\frac{1 - n}{2}}) \mathsf{Vertex}
\end{equation}
where $\chi_{\Lambda^k(\mathbb{C}^n)^*}$ is a character of the corresponding fundamental representation of the Langlands dual group $GL_n^\vee \simeq GL_n$. 
\end{theorem}

\begin{proof}
The proof is inductive in the rank $n$. From Theorem \ref{Xknmflop} for $m = n - 1$, we have an equality of localized $K$-theory classes
\begin{equation} \label{theoremapp}
    \chi\Big(k, \mathscr{W}_n \eval_0, \mathscr{V}_{n - 1} \eval_0 \Big) = \chi^\vee \Big(k, \mathscr{W}_n \eval_0, \mathscr{V}_{n - 1} \eval_0 \Big) + t^{n - k} \chi^\vee\Big(k - 1, \mathscr{W}_n \eval_0, \mathscr{V}_{n - 1} \eval_0 \Big)
\end{equation}
where $\chi^\vee$ denotes the class associated to the equivariant genus of $X^\vee$ in an analogous fashion. 

Now observe that, under $\langle \text{---} \rangle$, an insertion of $\chi^\vee(k, \mathscr{W}_n \eval_0, \mathscr{V}_{n - 1} \eval_0)$ may be reinterpreted as a certain twisted quasimap count, by an argument identical to the one leading to propostion \ref{heckeinsertion}. The count is defined by the moduli space 
\begin{equation}
    \{ \text{bundles $\mathscr{V}_i$, Hecke modification $0 \to \mathscr{V}^{-\mu}_{n - 1} \to \mathscr{V}_{n - 1} \to \mathscr{O}_0 \otimes \mathscr{E} \to 0$, stable section $f$} \}^\circ/ \sim
\end{equation}
where $f \in H^0(C, \mathscr{H}om(\mathscr{V}_1, \mathscr{V}_2) \oplus \dots \oplus \mathscr{H}om(\mathscr{V}_{n - 2}, \mathscr{V}_{n - 1}) \oplus \mathscr{H}om(\mathscr{V}^{-\mu}_{n - 1}, \mathscr{W}_n))$ , $-\mu$ is the opposite coweight of $\mu$ understood as the coweight of $GL_{n - 1}$ corresponding to $\Lambda^k(\mathbb{C}^{n - 1})$, and stability means that $f$ is generically injective as before. The superscript $\circ$ means quasimaps nonsingular at infinity, as usual. 

Because $\mathscr{V}_{n - 1}$ already varies in moduli, we are free to exchange the role of $\mathscr{V}_{n - 1}$ and $\mathscr{V}_{n - 1}^{-\mu}$ above. Thus, the count with $\chi^\vee$ inserted is a special instance of a count of the following data:
\begin{equation} \label{QMtwistati}
    \mathsf{QM}^{(i, \mu)}(X) := \{ \text{bundles $\mathscr{V}_j$, Hecke modification $0 \to \mathscr{V}_i \to \mathscr{V}^\mu_i \to \mathscr{O}_0 \otimes \mathscr{E} \to 0$, stable section $f$} \}/\sim
\end{equation}
where $f$ is now a global section of $\mathscr{H}om(\mathscr{V}_1, \mathscr{V}_2) \oplus \dots \oplus \mathscr{H}om(\mathscr{V}_{i - 1}, \mathscr{V}^\mu_i) \oplus \dots \oplus \mathscr{H}om(\mathscr{V}_{n - 1}, \mathscr{W}_n)$. Write $\mathsf{Vertex}^{(i, \mu_i)}$ for the normalized count of quasimaps nonsingular at $\infty \in C$ as in \eqref{normalizedvertex}. It is understood that we interpret $z^{\deg f}$ in the vertex function as $\prod_{j = 1}^n z_j^{-(\deg \mathscr{V}_j - \deg \mathscr{V}_{j - 1} + |\mu| \delta_{ij})}$. Taking into account the $z$-shifts, \eqref{theoremapp} under $\langle \text{---} \rangle$ may be rephrased as the equality of counts
\begin{equation} \label{vertexrecursion}
    \mathsf{Vertex}^{(n, \mu_k)} = \mathsf{Vertex}^{(n - 1, \mu_k)} + z_n^{-1} t^{n - k} \mathsf{Vertex}^{(n - 1, \mu_{k - 1})}. 
\end{equation}
On the other hand, by the same argument as in Proposition \ref{heckeinsertion}, we have 
\begin{equation} 
    \mathsf{Vertex}^{(n - 1, \mu_k)} = z_{n - 1}^{-k} \Big\langle \chi\Big( k, \mathscr{V}_{n - 1} \eval_0, \mathscr{V}_{n - 2} \eval_0 \Big) \Big\rangle. 
\end{equation}
Therefore, we may apply Theorem \ref{Xknmflop} in this fashion iteratively until we reach $n = 0$, in which case there is no Hecke modification at all and the count reduces to $\mathsf{Vertex}$ itself. It follows that $\mathsf{Vertex}^{(i, \mu_k)}$ is a scalar multiple of $\mathsf{Vertex}$ (for each $i$ and $k$), where the scalar depends on $z_1, \dots, z_i$ and $t$. 

To determine the scalar multiples, note that the characters satisfy the recursion 
\begin{equation} \label{characterrecursion}
\begin{split}
    t^{\frac{1}{2} \dim \text{Gr}(k, n)} \chi_{\Lambda^k(\mathbb{C}^n)^*}(z_i t^{\frac{n + 1 - 2i}{2}}) & = t^{k/2} \times t^{-k/2} \times t^{\frac{1}{2} \dim \text{Gr}(k, n - 1)} \chi_{\Lambda^k(\mathbb{C}^{n - 1})^*}(z_i t^{\frac{n - 2i}{2}}) \\ & + t^{(n - k)/2} \times z_n^{-1} t^{(n - 1)/2} \times t^{-(k - 1)/2} \times t^{\frac{1}{2} \dim \text{Gr}(k - 1, n - 1)} \chi_{\Lambda^{k - 1}(\mathbb{C}^{n - 1})^*}(z_i t^{\frac{n - 2i}{2}}) \\
    & = t^{\frac{1}{2} \dim \text{Gr}(k, n - 1)} \chi_{\Lambda^k(\mathbb{C}^{n - 1})^*}(z_i t^{\frac{n - 2i}{2}}) + z_n^{-1} t^{n - k} \times t^{\frac{1}{2} \dim \text{Gr}(k - 1, n - 1)} \chi_{\Lambda^{k - 1}(\mathbb{C}^{n - 1})^*}(z_i t^{\frac{n - 2i}{2}}). 
\end{split}
\end{equation}
This follows from the restriction formula $\Lambda^k(\mathbb{C}^n) = \Lambda^k(\mathbb{C}^{n - 1}) \oplus \Lambda^{k - 1}(\mathbb{C}^{n - 1})$, where both sides are understood as $GL_{n - 1} \subset GL_n$ representations, the fact that the characters are homogeneous of degree $-k$ in the $z$-variables, and trivial properties of $t^{\frac{1}{2} \dim \text{Gr}(k, n)} = t^{k(n - k)/2}$. Moreover, this recursion uniquely determines the characters as functions of $(z_1, \dots, z_n, t)$. Comparing \eqref{characterrecursion} and \eqref{vertexrecursion}, the theorem immediately follows by induction and an application of Proposition \ref{vertexcomparison}. 
\end{proof}

\section{Cohomological computations: quantized multiplication morphisms} \label{multmaps}
As our considerations are primarily geometric, they work equally well in equivariant cohomology or equivariant $K$-theory. However, there is a certain strengthening of Theorem \ref{heckeeigenval} available after passing to the $t \to 0$ and cohomological limit, which makes contact with the theory of multiplication morphisms for generalized affine Grassmannian slices introduced in \cite{bfnslice} and developed in \cite{finkelberg2017}, \cite{krylovperunov}.

\subsection{Quantized Coulomb branch algebra} \label{quantizedMc}
In this section we will introduce the quantized Coulomb branch algebra of pure gauge theory, and recall certain key features of it that will be relevant for our analysis of the cohomological vertex functions. Standard references on Coulomb branches and details on everything we say below are \cite{bfn}, \cite{bfnslice}. 

\subsubsection{}
From the definition \eqref{affinegrass} of the affine Grassmannian, it is clear that the group $\mathbb{C}^\times_\varepsilon \ltimes G(\mathscr{O})$ acts on it naturally by loop rotations and left multiplication. The equivariant Borel-Moore homology (with $\mathbb{C}$ coefficients for us, always)
\begin{equation}
    \widehat{\mathscr{M}}_C := H_*^{\mathbb{C}^\times_\varepsilon \ltimes G(\mathscr{O})}(\text{Gr}_G)
\end{equation}
is well-defined as an infinite-dimensional graded vector space, as is evident using the stratification of $\text{Gr}_G$ by the finite-dimensional $G(\mathscr{O})$-orbits. The convolution diagram of the affine Grassmannian gives it the structure of a noncommutative algebra as observed in \cite{bfm}, \cite{bfn} which is referred to as the quantized Coulomb branch algebra of pure gauge theory of type $G$. For us the reductive group is always $G = GL_n$ for some $n$. At the specialization $\varepsilon = 0$ of the equivariant parameter for the loop rotation, the algebra becomes commutative, and its spectrum 
\begin{equation}
    \mathscr{M}_C := \text{Spec} \, H_*^{G(\mathscr{O})}(\text{Gr}_G)
\end{equation}
is in fact a smooth symplectic variety of dimension $2n$. This variety is called the Coulomb branch of pure gauge theory with gauge group $G$, and the algebra $\widehat{\mathscr{M}}_C$ is a quantization of its algebra of global functions. Letting $\mathfrak{a} \subset \text{Lie} \, G$ denote a Cartan subalgebra, there is an obvious commutative subalgebra 
\begin{equation}
    \mathbb{C}[\mathfrak{a}]^W = H^\bullet_G(\text{pt)} \subset \widehat{\mathscr{M}}_C
\end{equation}
meaning that the representation theory of $\widehat{\mathscr{M}}_C$ is closely related to the study of quantum integrable systems. Our goal in this section is to introduce a new perspective on this relationship using vertex functions and nonabelian shift operators. 

\subsubsection{}
To understand $\widehat{\mathscr{M}}_C$, a productive point of view is to use equivariant localization by a maximal torus $\mathbb{C}^\times_\varepsilon \times A \subset \mathbb{C}^\times_\varepsilon \ltimes G(\mathscr{O})$. The fixed locus $\text{Gr}_G^{\mathbb{C}^\times_\varepsilon \times A} = \text{Gr}_A$, so the localization theorem gives rise to an injective homomorphism of convolution algebras 
\begin{equation}
    H_*^{\mathbb{C}^\times_\varepsilon \ltimes G(\mathscr{O})}(\text{Gr}_G) \xhookrightarrow{} H_*^{\mathbb{C}^\times_\varepsilon \times A}(\text{Gr}_A)^W_{\text{loc}} 
\end{equation}
where the superscript in the target denotes the subring of Weyl group invariants and the subscript denotes the need to localize at certain hyperplanes in the equivariant parameters. It will be technically convenient for us to localize at the generic point of equivariant parameters, at which point the target becomes isomorphic to the ring of symmetric meromorphic difference operators on $\mathfrak{a} = \text{Lie} \, A$. See Appendix A of \cite{bfnslice} for a very explicit discussion of this embedding. 

\subsubsection{}
It is possible to give a reasonably explicit set of generators for $\widehat{\mathscr{M}}_C$ as follows. Let $\mu^+ = (1, \dots, 0)$ and $\mu^- = (0, \dots, -1)$ be the fundamental and antifundamental minuscule dominant cocharacters for $GL_n$. The associated $G(\mathscr{O})$ orbits $\overline{\text{Gr}}^{\mu^\pm}_G = \text{Gr}^{\mu^\pm}_G = G(\mathscr{O}) z^{\mu^\pm}G(\mathscr{O}) \subset \text{Gr}_G$ are closed and isomorphic to projective spaces of dimension $n - 1$, and carry associated tautological quotient bundles $Q_\pm$ of rank $n - 1$. The characteristic classes of these bundles give rise to well-defined homology classes which we assemble into generating functions via the Chern polynomials
\begin{equation} \label{U+-ops}
\begin{split}
    \widehat{U}^+(x) & := [\overline{\text{Gr}}^{\mu^+}_{G}] \cap c_x(Q_+) \in H_*^{\mathbb{C}^\times_\varepsilon \ltimes G(\mathscr{O})}(\text{Gr}_G)[x] \\
    \widehat{U}^-(x) & := -[\overline{\text{Gr}}^{\mu^-}_G] \cap c_{-x}(Q_-) \in H_*^{\mathbb{C}^\times_\varepsilon \ltimes G(\mathscr{O})}(\text{Gr}_G)[x]
\end{split}
\end{equation}
The overall sign is for later convenience. If $(a_1, \dots, a_n)$ denote the equivariant parameters for the group $G$, we also have the generating function of the symmetric polynomials 
\begin{equation}
    Q(x) := \prod_i(x - a_i) \in H^\bullet_G(\text{pt})[x] \subset H_*^{\mathbb{C}^\times_\varepsilon \ltimes G(\mathscr{O})}(\text{Gr}_G)[x].
\end{equation}
Note that $\deg_x Q(x) = n$, $\deg_x \widehat{U}^\pm(x) = n - 1$, and $Q(x)$ is monic. 

\subsubsection{}
We have the following 
\begin{prop} \label{mcrel}
There exists a unique class $\widetilde{Q}(x) \in H_*^{\mathbb{C}^\times_\varepsilon \ltimes G(\mathscr{O})}(\text{Gr}_G)[x]$, $\deg_x \widetilde{Q}(x) = n - 2$, such that 
\begin{equation} \label{qdetrel}
    \widetilde{Q}(x  + \varepsilon/2) Q(x -\varepsilon/2) - \widehat{U}^+(x + \varepsilon/2) \widehat{U}^-(x - \varepsilon/2) = 1. 
\end{equation}
\end{prop}
In view of geometric Satake this should be thought of as analogous to the decomposition of the tensor product $V \otimes V^*$, viewed as a $GL(V)$-representation, into the trivial and adjoint representations. 

\begin{proof}
Under equivariant localization $\widehat{\mathscr{M}}_C \xhookrightarrow{} H_*
^{\mathbb{C}^\times_\varepsilon \times A}(\text{Gr}_A)^W_{\text{loc}}$, $\widehat{U}^\pm(x)$ map to the difference operators (which by abuse of notation we refer to also as $\widehat{U}^\pm$)
\begin{equation}
\widehat{U}^\pm(x) = \pm\sum_{i = 1}^n e^{\pm \varepsilon \frac{\partial}{\partial a_i}} \prod_{j (\neq i)} \frac{x - a_j}{a_i - a_j}.
\end{equation}
In our conventions, these act on functions from the right. We compute 
\begin{equation}
\begin{split}
    \widehat{U}^+(x)\widehat{U}^-(x - \varepsilon) & =  - \sum_{i, i'} e^{\varepsilon(\partial_i - \partial_{i'})} \prod_{j (\neq i)} \frac{a_j + \varepsilon \delta_{i'j} - x}{a_j - a_i + \varepsilon \delta_{i'j} - \varepsilon \delta_{ii'}} \prod_{j' (\neq i')} \frac{x - a_{j'} - \varepsilon}{a_{i'} - a_{j'}} \\
    & = -\sum_i \prod_{j (\neq i)} \frac{a_j - x}{a_j - a_i - \varepsilon} \prod_{j' (\neq i)} \frac{x - a_{j'} - \varepsilon}{a_i - a_{j'}} - \sum_{i \neq i'} e^{\varepsilon(\partial_i - \partial_{i'})} \prod_{j (\neq i)} \frac{a_j + \varepsilon \delta_{i' j} - x }{a_j - a_i + \varepsilon\delta_{i'j}} \prod_{j' (\neq i')} \frac{x - a_{j'} - \varepsilon}{a_{i'} - a_{j'}}. 
\end{split}
\end{equation}
When $x = a_\ell + \varepsilon$ for some $\ell$, only the $i = \ell$ term survives in the first sum, which contributes $-1$. Every term in the second sum turns out to vanish: if $i' \neq \ell$, then the product over $j'$ vanishes from the $j' = \ell$ term, and if $i' = \ell$ then the product over $j$ vanishes from the $j = \ell$ term. 

The right hand side of this equation is a difference operator, and the coefficients of the basic shift operators in $a_i$ are polynomials in $x$ with coefficients in the field of fractions of $H^\bullet_A(\text{pt})$. Applying the polynomial division algorithm (which is indeed applicable since we have coefficients in a field) to the coefficients of $e^{\varepsilon(\partial_i - \partial_{i'})}$, we conclude the existence of a unique meromorphic difference operator $\widetilde{Q}(x) \in H_*^{\mathbb{C}^\times_\varepsilon \times A}(\text{Gr}_A)^W_{\text{loc}}$ of degree $n - 2$ in $x$ such that 
\begin{equation}
    \widehat{U}^+(x) \widehat{U}^-(x - \varepsilon) = -1 + \widetilde{Q}(x) Q(x - \varepsilon). 
\end{equation}
All that remains is to show that $\widetilde{Q}(x)$ is in fact the image of a (necessarily unique) non-localized equivariant homology class under localization. Indeed, note that since $Q(x - \varepsilon)$ is monic, it has a multiplicative inverse $Q(x - \varepsilon)^{-1} \in H_*^{\mathbb{C}^\times_\varepsilon \ltimes G(\mathscr{O})}(\text{Gr}_G)[x, x^{-1} ]\!]$ obtained by expansion at $x \to \infty$. Then we just observe that
\begin{equation}
    (\widehat{U}^+(x) \widehat{U}^-(x - \varepsilon) + 1)Q(x - \varepsilon)^{-1} = \widetilde{Q}(x) \in H_*^{\mathbb{C}^\times_\varepsilon \times A}(\text{Gr}_A)^{W}_{\text{loc}}[x, x^{-1} ]\!]
\end{equation}
but the coefficients of powers of $x$ in the left hand side are the images under equivariant localization of classes in $H_*^{\mathbb{C}^\times_\varepsilon \ltimes G(\mathscr{O})}(\text{Gr}_G)$, while the right hand side is manifestly a polynomial in $x$. The only way this is possible is if both sides are in fact in the image of the injection $H_*^{\mathbb{C}^\times_\varepsilon \ltimes G(\mathscr{O})}(\text{Gr}_G)[x] \xhookrightarrow{} H_*^{\mathbb{C}^\times_\varepsilon \times A}(\text{Gr}_A)^W_{\text{loc}}[x]$. Shifting $x \mapsto x + \varepsilon/2$ gives the proposition as stated.
\end{proof}

The point of introducing the class $\widetilde{Q}(x)$ is the following 
\begin{prop} \label{Mcgen}
    The coefficients of $Q(x)$, $\widehat{U}^\pm(x)$, $\widetilde{Q}(x)$ generate the ring $H_*^{\mathbb{C}^\times_\varepsilon \ltimes G(\mathscr{O})}(\text{Gr}_G)$ over $\mathbb{C}[\varepsilon]$. 
\end{prop}
\begin{proof}
It suffices to show this at $\varepsilon = 0$, in which case it follows from Theorem 3.1 of \cite{bfnslice}. 
\end{proof}
We remark as a corollary of the proof of Proposition \ref{mcrel}, $\widetilde{Q}(x)$ is in fact uniquely determined through the relation \eqref{qdetrel}. In particular, any algebra homomorphism $\widehat{\mathscr{M}}_C \to \mathcal{A}$ is determined by the image of the coefficients of $\widehat{U}^\pm(x)$ and $Q(x)$.

\subsection{Vertex function and quantized multiplication maps} \label{vertexmult}
The purpose of this section will be to study the interaction of the cohomological vertex function of $X = GL_n/B$, which we continue to denote by $\mathsf{Vertex}(z, a)$, with the Coulomb branch algebra $\widehat{\mathscr{M}}_C$. By design, the algebra ``acts'' naturally on the vertex function via nonabelian shift operators. The goal of this section is to prove Theorem \ref{vertexint}, which is a more precise and refined version of this statement. 
\subsubsection{}
For convenience, we first remind the reader of some standard facts about the limits we consider. It is obvious that, when $t \to 0$, the equivariant Hirzebruch genus $\chi(X)$ simply computes the Euler character of the structure sheaf $\mathscr{O}_X$ of regular functions on $X$. The cohomological limit, from the standpoint of localization formulas, just means replacing 
\begin{equation}
    \prod_w \frac{1}{1 - w(a)^{-1}} \longrightarrow \prod_w \frac{1}{\text{weight}(\mathfrak{a})}.
\end{equation}
Note the $K$-theoretic weights $w(a)$ are functions on the torus $A \ni a$, while the cohomological weights $\text{weight}(\mathfrak{a})$ are functions on the Lie algebra $\text{Lie} \, A \ni \mathfrak{a}$. The limit may be understood as $a = e^{\beta \mathfrak{a}}$, $\beta \to 0^+$ and discarding an overall power of $\beta$. In summary, we have 
\begin{equation}
\chi(X) \xrightarrow[]{\, t \to 0 \, } \chi(X, \mathscr{O}_X) \xrightarrow[]{\, K \to H^\bullet \, } \int_X 1. 
\end{equation}
Note if $X$ is proper the integral of $1$ is in fact zero on degree grounds. When $X$ is not proper but has proper fixed locus under the torus action, the integral is defined by equivariant residue and takes values in localized equivariant cohomology $H^\bullet_A(\text{pt})_{\text{loc}}$. 

We will also frequently make use of Chern polynomials in this section, so we remind the reader of the notation: for a $G$-equivariant vector bundle $\mathscr{E}$ of rank $n$, 
\begin{equation}
    c_x(\mathscr{E}) := \sum_{i = 0}^n x^{n - i} c_i(\mathscr{E}) \in H^\bullet_G(X)[x].
\end{equation}

\subsubsection{}
While all the really important computations in this paper are for $X$ finite-dimensional, it is worth commenting on the only kind of infinite-dimensional situation we consider, namely when $X = \mathsf{Maps}(\mathbb{C}_q \to \mathbb{C}_x)$. In this case the procedure described above (with $q = e^{\beta \varepsilon}$) produces the divergent infinite product 
\begin{equation}
    \frac{1}{\varphi_q(x)} \to \prod_{n = 0}^\infty \frac{1}{-x -n\varepsilon} \xrightarrow[]{\zeta-\text{regularize}} \frac{1}{\sqrt{-2 \pi \varepsilon}} \times (-\varepsilon)^{x/\varepsilon} \Gamma\Big(\frac{x}{\varepsilon} \Big) := \Gamma_\varepsilon(x)
\end{equation}
where in the final arrow we have understood the product by $\zeta$-function regularization ($\Gamma(x)$ denotes the gamma function satisfying $\Gamma(x + 1) = x\Gamma(x)$). Such regularization is sensible because it preserves the difference equations we are interested in studying: we have $\Gamma_\varepsilon(x + \varepsilon) = -x \Gamma_\varepsilon(x)$, in direct parallel to the $q$-difference equation satisfied by $\varphi_q(x)$. 

Actual enumerative computations in Section \ref{kthdiff} are performed on finite-dimensional moduli spaces, so there is no ambiguity in taking their limit to cohomology, but we have seen the $\varphi_q$-functions enter as overall prefactors in formulas like \eqref{normalizedvertex}. We follow the convention established in this section and replace such factors by $\Gamma_\varepsilon(x)$ in the cohomological limit, compare to \cite{Iritani_2009}.

\subsubsection{}
In the $t \to 0$ and cohomological limit, the normalized vertex function \eqref{normalizedvertex} (for no shift, $\mu = 0$) passes over to 
\begin{equation}
    \mathsf{Vertex}(z, a) = \int_{\mathsf{QM}^\circ(X)} \text{ev}_\infty^*\Big(z^{\frac{c_1(V)}{\varepsilon}} \times \Gamma_\varepsilon(TX + \varepsilon) \Big) z^{\deg f}.
\end{equation}
Note that, for quasimaps to $X = GL_n/B$, the virtual fundamental class coincides with the usual fundamental class. In this equation, 
\begin{equation}
    z^{\frac{c_1(V)}{\varepsilon}} := \prod_{i = 1}^n z_i^{-\frac{c_1(V_i) - c_1(V_{i - 1})}{\varepsilon}}
\end{equation}
and the analytic $\Gamma_\varepsilon$-class may be understood as a genus via 
\begin{equation}
    \Gamma_\varepsilon(TX + \varepsilon) = \frac{1}{\sqrt{-2\pi \varepsilon}} \times \frac{\varepsilon}{\text{eu}(TX + \varepsilon)} \times \exp\Big( \frac{c_1(X) + \varepsilon\dim X }{\varepsilon}(\log (-\varepsilon) - \gamma) + \sum_{k = 2}^\infty \frac{(-1)^k}{k \varepsilon^k} \zeta(k) \text{ch}_k(TX + \varepsilon) \Big)
\end{equation}
where $\text{ch}_k$ is the degree $2k$ part of the Chern character and $\varepsilon$ is a trivial line bundle of weight $1$ under $\mathbb{C}^\times_\varepsilon$. 

\subsubsection{}
It will be very convenient to introduce the following generalizations of the vertex function. Consider the moduli spaces $\mathsf{QM}^{\circ, (i, \mu^\pm_i)}(X)$ as in \eqref{QMtwistati}, where $i \in \{1, \dots, n \}$ and $\mu^\pm_i$ is a fundamental or antifundamental cocharacter of $GL_i$ as in Section \ref{quantizedMc}. As usual, the superscript $\circ$ denotes quasimaps nonsingular at infinity. The moduli spaces have canonical maps 
\begin{equation}
    \mathsf{QM}^{\circ, (i, \mu^\pm_i)}(X) \xrightarrow{\, \, \, \pi_i \, \, \,} \overline{\text{Gr}}^{\mu^\pm_i}_{GL_i}. 
\end{equation}
We recall that the target has the rank $i - 1$ quotient bundle $Q_\pm^{(i)}$. Inspired by \eqref{U+-ops}, define
\begin{equation}
    \langle U^\pm_i(x) \rangle := \pm\int_{\mathsf{QM}^{\circ, (i, \mu^\pm_i)}(X)} \text{ev}_\infty^*\Big(z^{\frac{c_1(V)}{\varepsilon}} \times \Gamma_\varepsilon(TX + \varepsilon) \Big) \cup \pi_i^*(c_{\pm x}(Q_\pm^{(i)})) z^{\deg f}
\end{equation}
and likewise 
\begin{equation}
    \langle Q_i(x)\rangle := \int_{\mathsf{QM}^\circ(X)} \text{ev}_\infty^*\Big(z^{\frac{c_1(V)}{\varepsilon}} \times \Gamma_\varepsilon(TX + \varepsilon) \Big) \cup c_x \Big( \mathscr{V}_i \eval_0 \Big) z^{\deg f}
\end{equation}
with $\mathscr{V}_i$ one of the tautological bundles over $\mathsf{QM}(X) \times C$. We will use a similar notation $\langle f(\mathscr{V}) \rangle$ for descendent insertions of general Schur functors of tautological bundles restricted at $0 \in C$, in the terminology of \cite{okounkovpcmi}.

We may also consider moduli spaces of quasimaps to $X$ with successive Hecke modifications of the bundle $\mathscr{V}_i$, which are again smooth since $\mu^\pm_i$ are minusucle and have maps to convolution Grassmannians 
\begin{equation}
    \mathsf{QM}^{\circ, (i, \mu_i^+, \mu_i^-)}(X) \xrightarrow[]{\, \, \, \pi_i \, \, \, } \overline{\text{Gr}}^{\mu^+_i}_{GL_i} \widetilde{\times} \overline{\text{Gr}}^{\mu^-_i}_{GL_i}.
\end{equation}
The convolution Grassmannian has two universal quotient bundles $Q_+$, $Q_-$ of rank $i - 1$. Then using the convolution product we may define the following insertion (compare to the proof of Proposition \ref{mcrel}) 
\begin{equation}
\begin{split}
    \langle \widetilde{Q}_i(x) \rangle & := \langle(U_i^+(x) U_i^-(x - \varepsilon) + 1)Q_i(x - \varepsilon)^{-1} \rangle \\ & := \int_{\mathsf{QM}^{\circ, (i, \mu^+_i, \mu^-_i)}(X)} \text{ev}_\infty^*\Big(z^{\frac{c_1(V)}{\varepsilon}} \times \Gamma_\varepsilon(TX + \varepsilon) \Big) \cup \pi_i^*(1 - c_{x}(Q_+^{(i)}) \cup c_{- x + \varepsilon}(Q_-^{(i)})) \cup \frac{1}{c_{x - \varepsilon}( \mathscr{V}_i \eval_0 )} z^{\deg f}
\end{split}
\end{equation}
which is in fact a polynomial of degree $n - 2$ in $x$ as a corollary of the proof of Proposition \ref{mcrel}. It is important that insertions of $U^\pm$ do not commute, because the convolution product on $\text{Gr}_G$ is not commutative $\mathbb{C}^\times_\varepsilon$ equivariantly. In analogy with correlation functions in quantum mechanics, it is best to view insertions as time-ordered from right to left, with the insertions of $U^\pm_i(x)$ signifying the order in which we perform modifications of the bundle $\mathscr{V}_i$. 

We understand that, in the presence of a shift by $\mu^\pm_i$,
\begin{equation}
    z^{\deg f} = \prod_{j = 1}^nz_j^{-(\deg \mathscr{V}_j - \deg \mathscr{V}_{j -  1} \pm \delta_{ij})}.
\end{equation}

\subsubsection{}
Now let 
\begin{equation}
    \widehat{U}^\pm_n(x) := \pm \sum_{i = 1}^n e^{\pm \varepsilon \frac{\partial}{\partial a_i}} \prod_{j (\neq i)} \frac{x - a_j}{a_i - a_j}
\end{equation}
denote the familiar difference operators, and write $Q_n(x) = \prod_i(x - a_i) = c_x(\mathscr{W}_n \eval_0)$,  $\widetilde{Q}_n(x) = (\widehat{U}^+_n(x) \widehat{U}^-_n(x - \varepsilon) + 1)Q_n(x - \varepsilon)^{-1}$ as usual. Finally, let 
\begin{equation}
    \mathsf{Vertex}^\theta(z, a) := \mathsf{Vertex}( (-1)^{i - 1}z_i, a)
\end{equation}
denote the vertex with sign of the $z$ variables changed as $(z_1, z_2, \dots, z_n) \mapsto (z_1, -z_2, \dots, (-1)^{n - 1}z_n)$. 

We are now in a position to prove 
\begin{theorem} \label{vertexint}
The vertex function satisfies difference equations in the $a$-variables and differential equations in the $z$-variables characterized by 
\begin{equation} 
\begin{bmatrix}x - \varepsilon z_n \partial_{z_n} & z_n^{-1} \\ -z_n & 0 \ \end{bmatrix} \times \dots \times \begin{bmatrix} x - \varepsilon z_1 \partial_{z_1} & z_1^{-1} \\ -z_1 & 0 \end{bmatrix} \cdot \mathsf{Vertex}^\theta(z, a) = \mathsf{Vertex}^\theta(z, a) \cdot \begin{bmatrix} Q_n(x) & \widehat{U}^+_n(x) \\ \widehat{U}^-_n(x) & \widetilde{Q}_n(x) \end{bmatrix}
\end{equation}
where we understand that differential operators in $z$ act from the left and difference operators in $a$ act from the right. 
\end{theorem}

\begin{proof}
First, by reasoning identical to the proof of Propositions \ref{vertexcomparison}, we deduce 
\begin{equation}
    \mathsf{Vertex}^\theta(z, a) \cdot \widehat{U}^\pm_n(x) = \langle U^\pm_n(x) \rangle
\end{equation}
with the notation for the right hand side understood as above. The strategy of the proof is very simple: we will arrive at an inductive formula for the action of the operators on the right hand side. For the convenience of readers we break this into three steps and consider the action of the operators one at a time. 

\subsubsection{}
Let us first concentrate on the case of $\langle U^-_n(x) \rangle$. Spelling out the $\mathbb{C}^\times_\varepsilon \times A$-equivariant localization formula as in Proposition \ref{heckeinsertion}, we find
\begin{equation}
    \langle U^-_n(x) \rangle = (-1)^n z_n\Bigg\langle \sum_{i = 1}^n  \prod_\ell (v_\ell - a_i) \prod_{j (\neq i)} \frac{x - a_j}{a_i - a_j}\Bigg\rangle
\end{equation}
where $v_\ell$ are the Chern roots of $\mathscr{V}^*_{n - 1} \eval_0$ and $a_i$ are the equivariant variables for $A$, which are also the Chern roots of $\mathscr{W}^*_n \eval_0$. Now by the interpolation formula, the insertion is equal to 
\begin{equation}
    \sum_{i = 1}^n  \prod_\ell (v_\ell - a_i) \prod_{j (\neq i)} \frac{x - a_j}{a_i - a_j} = \prod_\ell(v_\ell - x) = (-1)^{n - 1}c_x \Big(\mathscr{V}_{n - 1} \eval_0 \Big)
\end{equation}
which implies the equality of vertex functions with insertions 
\begin{equation} \label{U-recursion}
    \langle U^-_n(x) \rangle = -z_n \langle Q_{n - 1}(x) \rangle
\end{equation}
with notations as above. 

\subsubsection{}
Next we look for a formula for $\mathsf{Vertex}^\theta(z, a) Q_n(x) =\langle Q_n(x) \rangle$. Interpolating $Q_n(x)$ as a degree $n$ polynomial in $x$ at the $n - 1$ Chern roots $v_\ell$ leads to the equality of insertions
\begin{equation}
    \langle Q_n(x) \rangle = \Bigg\langle \Big( x - \sum_ia_i + \sum_\ell v_\ell \Big)c_x \Big(\mathscr{V}_{n - 1} \eval_0 \Big) + \sum_{\ell = 1}^{n - 1} \prod_i(v_\ell - a_i) \prod_{k (\neq \ell)}\frac{x - v_k}{v_\ell - v_k} \Bigg\rangle.
\end{equation}
Now as in the proof of Theorem \ref{heckeeigenval}, we recognize in the second term the localization formula for a count of sections of $\mathscr{H}om(\mathscr{V}^{\mu^+}_{n - 1}, \mathscr{W}_n)$ for a Hecke modification of $\mathscr{V}_{n - 1}$ of type $\mu^+$. Just as in the proof of Theorem \ref{heckeeigenval}, we may exchange the role of $\mathscr{V}_{n - 1}$ and $\mathscr{V}_{n - 1}^{\mu^+}$ to trade it for a count $\langle U^-_{n - 1}(x) \rangle$ in our above notation. This leads to 
\begin{equation}
     \langle Q_n(x) \rangle = \Bigg\langle \Big( x + c_1\Big(\mathscr{W}_n \eval_0 \Big) - c_1\Big( \mathscr{V}_{n - 1} \eval_0  \Big) \Big)c_x \Big(\mathscr{V}_{n - 1} \eval_0 \Big) \Bigg\rangle + (-1)^{n - 2} (-1)^{n - 1} (-1) z_n^{-1} \langle U^-_{n - 1}(x) \rangle. 
\end{equation}
The prefactors arise as follows: $(-1)^{n - 1} z_n^{-1}$ is from the degree shift in the count of sections of $\mathscr{H}om(\mathscr{V}^{\mu^+}_{n - 1}, \mathscr{W}_n)$, $(-1)^{n - 2}$ is a sign from replacing integration over the orbit $\overline{\text{Gr}}^{\mu^+}_{GL_{n - 1}}$ by its dual $\overline{\text{Gr}}^{\mu^-}_{GL_{n - 1}}$, and the extra $(-1)$ is from the definition of $U^-$. 

Now for any of the tautological bundles $\mathscr{V}_i$ over $\mathsf{QM}^\circ(X) \times C$ we have
\begin{equation}
    \deg \mathscr{V}_i = (\mathsf{QM}^\circ \times C \to \mathsf{QM})_* c_1(\mathscr{V}_i) = \int_C c_1(\mathscr{V}_i) = \frac{c_1(\mathscr{V_i} \eval_0) - \text{ev}^*_\infty c_1(V_i)}{\varepsilon} \in H^0(\mathsf{QM}^\circ ; \mathbb{Z}). 
\end{equation}
Then taking into account the factors $\text{ev}^*_\infty (z^{c_1(V)/\varepsilon}) z^{\deg f}$ in $\mathsf{Vertex}^\theta$, we may rewrite the above as 
\begin{equation} \label{Qrecursion}
    \langle Q_n(x) \rangle = (x - \varepsilon z_n \partial_{z_n}) \langle Q_{n - 1}(x) \rangle + z_n^{-1} \langle U^-_{n - 1}(x) \rangle. 
\end{equation}
This completes the second step. 

\subsubsection{}
Now we study the localization formula for $\langle U^+_n(x) \rangle$. We have
\begin{equation}
\begin{split}
    \langle U^+_n(x) \rangle & = (-1)^{n - 1} z_n^{-1} \Bigg\langle \sum_{i = 1}^n \frac{1}{\prod_\ell(v_\ell - a_i + \varepsilon)} \prod_{j (\neq i)} \frac{x - a_j}{a_i - a_j} \Bigg\rangle \\
    & = -z_n^{-1} \Bigg\langle \frac{1}{c_{x - \varepsilon}(\mathscr{V}_{n - 1} \eval_0)}\sum_{\ell = 1}^{n - 1} \prod_{k (\neq \ell)} \frac{x - v_k - \varepsilon}{v_\ell - v_k} \prod_j\frac{x - a_j}{v_\ell - a_j + \varepsilon} \Bigg\rangle + z_n^{-1} \Bigg\langle \frac{1}{c_{x - \varepsilon}(\mathscr{V}_{n - 1} \eval_0)} \Bigg\rangle.
\end{split}
\end{equation}
The second line follows from the first either by a wall-crossing formula as in Section \ref{geometricdiffeq} or explicitly by analysis of residues in the contour integral 
\begin{equation}
    \int_\gamma \frac{du}{2\pi i } \prod_\ell \frac{x - v_\ell - \varepsilon}{u - v_\ell -  \varepsilon} \prod_j \frac{x - a_j}{u - a_j} \frac{1}{x - u}. 
\end{equation}
By the standard logic, we identify the first term above with a count of sections of $\mathscr{H}om(\mathscr{V}_{n - 1}^{\mu^-}, \mathscr{W}_n)$. Exchanging the role of $\mathscr{V}_{n - 1}$ and $\mathscr{V}^{\mu^-}_{n - 1}$ as in the proof of Theorem \ref{heckeeigenval}, this becomes an insertion of $U^+_{n - 1}(x - \varepsilon)$:
\begin{equation}
    \langle U^+_n(x) \rangle = \Bigg\langle \frac{1}{c_{x - \varepsilon}(\mathscr{V}_{n - 1}^{\mu^+} \eval_0)} U^+_{n -1}(x - \varepsilon)  c_x \Big( \mathscr{W}_n \eval_0 \Big)  \Bigg\rangle + z_n^{-1} \Bigg\langle \frac{1}{c_{x-\varepsilon}(\mathscr{V}_{n - 1} \eval_0)}\Bigg\rangle
\end{equation}
where we multiplied the first term by $(-1)^{n - 1} z_n (-1)^{n - 2}$ to account for the degree shift and exchanging the orbit $\overline{\text{Gr}}^{\mu^-}_{GL_{n - 1}}$ by its dual. Note also that we must replace $\mathscr{V}_{n - 1}$ by $\mathscr{V}^{\mu^+}_{n - 1}$ in insertions. The reader may enjoy verifying that 
\begin{equation}
    \Bigg\langle \frac{1}{c_{x - \varepsilon}(\mathscr{V}^{\mu^+}_{n - 1} \eval_0)} U^+_{n - 1}(x - \varepsilon) \dots \Bigg \rangle = \Bigg\langle U^+_{n - 1}(x) \frac{1}{c_x(\mathscr{V}_{n - 1} \eval_0)} \dots \Bigg \rangle
\end{equation}
where dots denote an arbitrary insertion. Now we observe that the proof of \eqref{Qrecursion} implies that it continues to hold with arbitrary additional operator insertions to the left, \textit{provided we replace $\mathscr{V}_{n - 1}$ by $\mathscr{V}_{n - 1}^{\mu^-}$ wherever it appears in the term with $U^-_{n - 1}(x)$}. In all, we conclude 
\begin{equation}
\begin{split}
    \langle U^+_n(x) \rangle & = (x - \varepsilon z_n \partial_{z_n}) \langle U^+_{n - 1}(x) Q_{n - 1}(x)^{-1} Q_{n - 1}(x) \rangle + z_n^{-1} \langle U^+_{n - 1}(x) c_x\Big( \mathscr{V}_{n - 1}^{\mu^-} \eval_0 \Big)^{-1} U^-_{n - 1}(x) \rangle \\ & + z_n^{-1} \Bigg\langle \frac{1}{Q_{n - 1}(x - \varepsilon)} \Bigg\rangle.
\end{split}
\end{equation}
Now the reader may also verify that $\langle \dots c_x(\mathscr{V}^{\mu^-}_{n - 1} \eval_0)^{-1} U^-_{n - 1}(x) \rangle = \langle \dots U^-_{n - 1}(x - \varepsilon) c_{x - \varepsilon}(\mathscr{V}_{n - 1} \eval_0)^{-1} \rangle$. Then we have 
\begin{equation} \label{U+recursion}
\begin{split}
    \langle U^+_n(x) \rangle & = (x - \varepsilon z_n \partial_{z_n}) \langle U^+_{n - 1}(x) \rangle + z_n^{-1}\langle (U^+_{n - 1}(x)U^-_{n - 1}(x - \varepsilon) + 1) Q_{n - 1}(x - \varepsilon)^{- 1} \rangle \\
    & = (x - \varepsilon z_n \partial_{z_n}) \langle U^+_{n - 1}(x) \rangle + z_n^{-1} \langle \widetilde{Q}_{n - 1}(x) \rangle. 
\end{split}
\end{equation}
This completes the third step. 

\subsubsection{}
The recursion relation for $\widetilde{Q}_n(x)$, by construction, is uniquely determined through those for $Q_n(x)$ and $U^\pm_n(x)$. Combining \eqref{U-recursion}, \eqref{Qrecursion}, \eqref{U+recursion}, we conclude 
\begin{equation}
    \begin{bmatrix} x - \varepsilon z_n \partial_{z_n} & z_n^{-1} \\ -z_n & 0 \end{bmatrix} \begin{bmatrix} \langle Q_{n - 1}(x) \rangle & \langle U^+_{n - 1}(x) \rangle \\ \langle U^-_{n - 1}(x) \rangle & \langle \widetilde{Q}_{n -1}(x) \rangle\end{bmatrix} = \begin{bmatrix} \langle Q_n(x) \rangle & \langle U^+_n(x) \rangle \\ \langle U^-_n(x) \rangle & \langle \widetilde{Q}_n(x) \rangle \end{bmatrix} = \mathsf{Vertex^\theta}(z, a) \cdot \begin{bmatrix} Q_n(x) & \widehat{U}^+_n(x) \\ \widehat{U}^-_n(x) & \widetilde{Q}_n(x) \end{bmatrix}.
\end{equation}
Now, clearly we may repeat this process inductively until we reach $n = 0$. This gives the statement of the theorem. 
\end{proof}

\subsection{Coulomb branches and shifted Yangians} \label{shiftedyang}
It has been well-known in representation theory for some time that quantized Coulomb branches in their various incarnations arise as quotients of quantum groups. A very productive way to think about quantum groups is to use the so-called $RTT$ formalism. In this section we will explain how Theorem \ref{vertexint} may be viewed as the geometric basis for the $RTT$ formalism of the quantized Coulomb branch $\widehat{\mathscr{M}}_C = H_*^{\mathbb{C}^\times_\varepsilon \ltimes G(\mathscr{O})}(\text{Gr}_G)$ of pure gauge theory (for $G = GL_n$). 

The equivariant/quantization parameter $\varepsilon$ will be considered fixed and generic throughout this section. 

\subsubsection{}
Let $A^\vee \subset GL_n^\vee$ denote the maximal torus of the Langlands dual group to $G = GL_n$. Let $z = (z_1, \dots, z_n) \in \, A^\vee$ be coordinates, and let $\text{Diff}(A^\vee)$ denote the algebra of polynomial differential operators on $A^\vee$. 

By Proposition \ref{Mcgen}, we can read off from Theorem \ref{vertexint} that there is an embedding 
\begin{equation}
    \widehat{\mathscr{M}}_C \xhookrightarrow{} \text{Diff}(A^\vee)
\end{equation}
simply by viewing $\mathsf{Vertex}^\theta(z, a)$ as an explicit kernel for it. We will now use the description of $\widehat{\mathscr{M}}_C$ by differential operators to characterize it as a quotient of a certain quantum group known as a shifted Yangian. 

\subsubsection{}
We recall that the Yangian $\mathsf{Y}(\mathfrak{gl}_2)$ is an associative Hopf algebra over $\mathbb{C}$ defined as follows (we follow \cite{molev1994}). Let $V \simeq \mathbb{C}^2$ be the defining representation of $\mathfrak{gl}_2$. The fundamental object is the $R$-matrix
\begin{equation} \label{Rmatrix}
    R(x) = \frac{x - \varepsilon P_{12}}{x - \varepsilon} \in \text{End}(V_1 \otimes V_2)(x)
\end{equation}
where $P_{12}$ denotes permutation of the two tensor factors. The generators $T^{(r)}_{ij}$, $1 \leq i, j \leq 2$ of $\mathsf{Y}(\mathfrak{gl}_2)$ are encoded in formal power series 
\begin{equation}
    T_{ij}(x) = \delta_{ij} + \sum_{r = 1}^\infty \frac{T^{(r)}_{ij}}{x^r} \in \mathsf{Y}(\mathfrak{gl}_2)[\![ x^{-1} ]\!].
\end{equation}
These may be viewed as the entries of a matrix $T(x) \in \mathsf{Y}(\mathfrak{gl}_2)[ \![ x^{-1}]\!] \otimes \text{End}(V)$. Then the defining relations of $\mathsf{Y}(\mathfrak{gl}_2)$ are the $RTT$ relations 
\begin{equation} \label{RTT}
    R_{12}(x_1 - x_2)T_1(x_1) T_2(x_2) = T_2(x_2) T_1(x_1) R_{12}(x_1 - x_2) \in \mathsf{Y}(\mathfrak{gl}_2) \otimes \text{End}(V_1 \otimes V_2)[\![ x_1, x_1^{-1}, x_2, x_2^{-1} ]\!]
\end{equation}
where a subscript on $T$ means it acts in the corresponding tensor factor. 

The center of $\mathsf{Y}(\mathfrak{gl}_2)$ is generated \cite{molev1994} by the coefficients of the ``quantum determinant''
\begin{equation}
    \text{qdet} \, T(x) := T_{22}(x) T_{11}(x - \varepsilon) - T_{12}(x) T_{21}(x - \varepsilon) 
\end{equation}
and the Yangian $\mathsf{Y}(\mathfrak{sl}_2)$ may be identified with the quotient of $\mathsf{Y}(\mathfrak{gl}_2)$ by $\text{qdet} \, T(x) = 1$. 

\subsubsection{}
For applications to Coulomb branches, it is too restrictive to consider $T$-matrices with leading coefficient $1$ in the $x \to \infty$ expansion. Instead, one needs the following generalization referred to as an (antidominantly) shifted Yangian. The $RTT$ formalism for these that we invoke here was developed in \cite{Frassek_2022}.

Fix a dominant coweight $\mu = (\mu_1, \mu_2)$ of $GL_2$. Following \cite{Frassek_2022}, we consider more general $T$-matrices 
\begin{equation}
    T_{ij}(x) = \sum_{r \in \mathbb{Z}} \frac{T^{(r)}_{ij}}{x^r}
\end{equation}
which are assumed to satisfy the $RTT$ relations \eqref{RTT} (with the same $R$-matrix) in addition to the following condition. We require that $T(x)$ admits a Gauss decomposition 
\begin{equation} \label{gauss}
    T(x) = \begin{bmatrix} 1 & 0 \\ f(x) & 1\end{bmatrix} \begin{bmatrix} g_1(x) & 0 \\ 0 & g_2(x) \end{bmatrix} \begin{bmatrix} 1 & e(x) \\ 0 & 1\end{bmatrix}
\end{equation}
where $f(x) = O(x^{-1})$, $e(x) = O(x^{-1})$, and $g_i(x) = x^{\mu_i} + O(x^{\mu_i - 1})$ in the $x \to \infty$ expansion. The quantum group generated by the coefficients of such a $T$-matrix is denoted $\mathsf{Y}_{-\mu}(\mathfrak{gl}_2)$ and called an antidominantly shifted Yangian. When $\mu = 0$ we recover the usual $\mathsf{Y}(\mathfrak{gl}_2)$. 

The quantum determinant again generates \cite{Frassek_2022} the center of $\mathsf{Y}_{-\mu}(\mathfrak{gl}_2)$, and $\mathsf{Y}_{-\mu}(\mathfrak{sl}_2)$ may be defined as the quotient of $\mathsf{Y}_{-\mu}(\mathfrak{gl}_2)$ by $\text{qdet} \, T(x) = 1$. 

\subsubsection{}
Now we consider the matrix featuring on the left hand side of Theorem \ref{vertexint}:
\begin{equation}
    S(x) = \begin{bmatrix} x - \varepsilon z\partial_z & z^{-1} \\ -z & 0 \end{bmatrix} \in \text{Diff}(\mathbb{C}^\times) \otimes \text{End}(V).
\end{equation}
The following is a short computation.
\begin{prop}
    The matrix $S(x)$ satisfies the relation $$R_{12}(x_1 - x_2) S_1(x_1) S_2(x_2) = S_2(x_2) S_1(x_1) R_{12}(x_1 - x_2) $$
with the $R$-matrix as in \eqref{Rmatrix}. 
\end{prop}
In addition, inspecting the Gauss decomposition \eqref{gauss} of $S(x)$, we read off that $g_1(x) = x + \dots $ and $g_2(x) = x^{-1} + \dots $ where dots denote lower terms in the $x \to \infty$ expansion. Since $\text{qdet} \, S(x) = 1$, we conclude at once 

\begin{prop}
Let $\alpha = (1, - 1)$ denote the simple positive coroot of $SL_2$. Then there is a map 
\begin{equation}
    \mathsf{Y}_{-\alpha}(\mathfrak{sl}_2) \to \text{Diff}(\mathbb{C}^\times)
\end{equation}
uniquely determined by $T(x) \mapsto S(x)$. 
\end{prop}

\subsubsection{}
One of the advantages of the $RTT$ formalism is that the Hopf algebra structure on the quantum group is manifest. In particular, for any splitting $\mu = \mu_1 + \mu_2$ with $\mu_i$ both dominant, there is a coproduct (Proposition 2.136 of \cite{Frassek_2022}) 
\begin{equation}
    \mathsf{Y}_{-\mu}(\mathfrak{gl}_2) \xrightarrow{\, \, \, \Delta \, \, \, } \mathsf{Y}_{-\mu_1}(\mathfrak{gl}_2) \otimes \mathsf{Y}_{-\mu_2}(\mathfrak{gl}_2)
\end{equation}
given by $\Delta(T_{ij}(x)) = \sum_k T_{ik}(x) \otimes T_{kj}(x)$. Moreover the coproduct preserves the $\text{qdet} = 1$ condition so descends to a coproduct on the $\mathfrak{sl}_2$ shifted Yangian. 

In particular, for any $n \geq 1$ we observe that we have a map $\mathsf{Y}_{-n\alpha}(\mathfrak{sl}_2) \to \text{Diff}(A^\vee)$ (recall $A^\vee \simeq (\mathbb{C}^\times)^n$) uniquely determined by 
\begin{equation} \label{Mcyangian}
    T(x) \mapsto S(x) = \begin{bmatrix}x - \varepsilon z_n \partial_{z_n} & z_n^{-1} \\ -z_n & 0 \ \end{bmatrix} \times \dots \times \begin{bmatrix} x - \varepsilon z_1 \partial_{z_1} & z_1^{-1} \\ -z_1 & 0 \end{bmatrix}.
\end{equation}
Note this map factors as an iterated application of the coproduct followed by the tensor product of the maps $\mathsf{Y}_{-\alpha}(\mathfrak{sl}_2) \to \text{Diff}(\mathbb{C}^\times)$ constructed above. 

Since the matrix coefficients of $T(x)$ generate the shifted Yangian by definition and the matrix coefficients of $S(x)$ generate the image of the Coulomb branch under $\widehat{\mathscr{M}}_C \xhookrightarrow{} \text{Diff}(A^\vee)$ by Proposition \ref{Mcgen} and Theorem \ref{vertexint}, we conclude 

\begin{theorem} \label{yangianquotient}
The assignment \eqref{Mcyangian} uniquely determines a surjection $\mathsf{Y}_{-n\alpha}(\mathfrak{sl}_2) \twoheadrightarrow \widehat{\mathscr{M}}_C$. 
\end{theorem}
Note that we are able to conclude this without explicitly characterizing the kernel, though our method shows that the kernel is identified naturally with the kernel of the representation $\mathsf{Y}_{-n \alpha}(\mathfrak{sl}_2) \to \text{Diff}(A^\vee)$. The commuting family of differential operators in the matrix coefficient $S_{11}(x)$ are nothing but the Hamiltonians of the quantum open Toda chain, so in this way we have recovered and geometrized the quantum inverse scattering method for the open Toda system \cite{sklyanin2000}. 

\newpage
\begin{appendices}

\section{Physics motivation} \label{phys}
As mentioned in the introduction, the physics background for this paper is concerned with certain localization computations in a three-dimensional topologically twisted gauge theory (or a four-dimensional gauge theory, when computations are in $K$-theory). The twisted version of the theory is sometimes referred to as the 3d $A$-model, see \cite{gammagehilburn}, \cite{gammagehilburnmazelgee} for mathematical foundations in the case of abelian gauge theories.

\subsubsection{}
Precisely, the situation we consider is the following. We consider three-dimensional $\mathscr{N} = 4$ gauge theory with gauge group $U(n)$ in the $A$-type topological twist. The bosonic field content of this theory consists of a gauge field $A$, a real scalar $\sigma$, and a complex scalar $\varphi$ all valued in $\mathfrak{gl}_n$. 

In a companion paper \cite{tam24}, we studied localization computations in such theories in the presence of boundary conditions and $\Omega$-deformation in some detail. We consider in this paper the following generalization of the setup there. The worldvolume of our 3d gauge theory is $\mathbb{R}^2_\varepsilon \times I$, where $I = [0, 1]$ is an interval with coordinate $t$ and $\varepsilon$ denotes $\Omega$-background. At the $t = 0$ end of $I$, we place a Dirichlet boundary condition. At the $t = 1$ end of $I$, we place a Neumann boundary condition. 

On a Neumann boundary, the restrictions $A \eval_\partial$, $\varphi \eval_\partial$ and their fermionic partners comprise a 2d $\mathscr{N} = (2, 2)$ vector multiplet, and the three-dimensional topological supercharge restricts to that of the two-dimensional $A$-model on the boundary. Then given any two-dimensional $\mathscr{N} = (2, 2)$ degrees of freedom with $U(n)$ flavor symmetry, we can gauge this symmetry using the restrictions of the bulk vector multiplet fields and produce a nontrivial bulk/boundary system. 

In this paper in Section \ref{multmaps}, we consider precisely the situation where the boundary degrees of freedom are a gauged linear sigma model which flows in the Higgs phase to a nonlinear sigma model with target $X = GL_n/B$.

\subsubsection{}
On the one hand, with no bulk operator insertions the three-dimensional dynamics becomes essentially trivial in the presence of the dual Dirichlet and Neumann boundary conditions, in the sense that the moduli space of solutions to the appropriate BPS equations is a single point with no automorphisms. 

On the boundary, the BPS equations are vortex equations, whose moduli space of solutions is expected to coincide with the moduli space of quasimaps to $X$. The coupling between the bulk and boundary just asserts that the fixed trivial framing bundle in the quasimap problem should be identified, as a holomorphic bundle on the spatial slice $\mathbb{R}^2 \simeq \mathbb{C} \simeq C \setminus \{ \infty \}$, with the fixed trivial bundle arising from the gauge theory in the bulk. 

Consequently, the partition function on $\mathbb{R}^2_\varepsilon \times I$ with the flag variety and Dirichlet boundary conditions just computes the vertex function of the flag variety. The $a$-variables may be interpreted as the boundary values of the bulk vector multiplet scalar $\varphi$ on the Dirichlet boundary. 

\subsubsection{}
In the presence of insertions of a bulk monopole operator of charge $\mu$, the bulk BPS equations require that the gauge bundle restricted on the Neumann boundary is related to the fixed trivial one on the Dirichlet boundary by a Hecke modification of type $\mu$. Then the full moduli space of solutions to the BPS equations coincides with one of the spaces $\mathsf{QM}^{(\circ, \mu)}(X)$ considered at length in the main text. 

On the other hand, the explicit localization formula for the Hecke modified vertex function relates the bulk monopole operator insertions to the action of a difference operator on the partition function without an insertion. This is the content of the cohomological limit of Proposition \ref{vertexcomparison}.

\subsubsection{}
The strategy of the proof of Theorem \ref{vertexint} has the following interpretation in 3d gauge theory. Because we use the gauged linear sigma model description of the boundary degrees of freedom, there is in fact a gauge group $U(1) \times U(2) \times \dots \times U(n - 1)$ living at the Neumann boundary. This can be ``unfolded'' in an accordion-like way into a sequence of $U(1), U(2), \dots, U(n - 1)$ 3d gauge theories with walls/interfaces between them. On either side of the interface we place Neumann boundary conditions on the vector multiplets and couple them to a 2d $\mathscr{N} = (2, 2)$ chiral multiplet that lives in the bifundamental representation $\text{Hom}(\mathbb{C}^{k - 1}, \mathbb{C}^k)$ of $U(k - 1) \times U(k)$. 

We can now take this rather complicated configuration and try to follow its renormalization group flow to the Coulomb branch. Then each segment with only a bulk $U(k)$ vector multiplet becomes well-approximated by a nonlinear sigma model to the Coulomb branch $\mathscr{M}_C(GL_k) = \text{Spec} \, H_*^{GL_k(\mathscr{O})}(\text{Gr}_{GL_k})$. On very general grounds, an interface becomes in this language a holomorphic Lagrangian correspondence defined in the product $\mathscr{M}_C(GL_{k - 1}) \times \mathscr{M}_C(GL_k)$. 

Because this Lagrangian has a simple description in the gauge theory in the ultraviolet, we are fortunate enough to identify it as cut out by some explicit equations. Identification and application of these equations constitutes the inductive step of the proof of Theorem \ref{vertexint}. Since we work in the $\Omega$-background, the holomorphic Lagrangian is replaced by its quantized version, namely a bimodule for the quantized Coulomb branches $\widehat{\mathscr{M}}_C(GL_{k - 1})$ and $\widehat{\mathscr{M}}_C(GL_k)$. In the proof of Theorem \ref{vertexint} we see the elements of those Coulomb branches entering as explicit cohomology classes pulled back from $GL_k(\mathscr{O})$-orbits on the $GL_k$ affine Grassmannian. 

\subsubsection{}
Moreover, the fact that the key ingredient in the identification of these equations (and their $K$-theoretic analogs) is a wall-crossing formula from Section \ref{geometricdiffeq} is no coincidence. The $X(k, n, m)$ quiver varieties studied in that section have a real FI parameter $\zeta_{\mathbb{R}}$ whose sign determines the stability condition for the GIT quotient, in other words distinguishes $X(k, n, m)$ and $X^\vee(k, n, m)$. The quiver description of $X(k, n, m)$ in fact arises as the effective theory living on the monopole operator, viewed as a probe of the bulk theory near the interface. The sign of the stability parameter determines whether the monopole operator is inserted to the left or the right of the interface, and thus the wall-crossing formulas provide the bridge which systematically moves such insertions across the interfaces. 

\subsubsection{}
As the algebra of polynomial differential operators $\text{Diff}(A^\vee) \simeq \widehat{\mathscr{M}}_C(A)$, we see that Theorem \ref{vertexint} may be rephrased as the assertion that the flag variety boundary condition may in fact be promoted to an interface between the $A^c$ gauge theory (superscript denotes maximal compact) and $U(n)$ gauge theory. In the semiclassical limit $\varepsilon \to 0$ this interface becomes a Lagrangian correspondence which is in fact the graph of an open embedding $\mathscr{M}_C(A) \xhookrightarrow{} \mathscr{M}_C(GL_n)$ known as the Toda embedding. This embedding is also the simplest instance of a multiplication morphism for generalized affine Grassmannian slices \cite{bfnslice}, \cite{finkelberg2017}, \cite{krylovperunov}, in accordance with a general conjecture due to J. Hilburn. 

\newpage 

\section{Flops and contour integrals} \label{contourint}
In this section we explain a different point of view on Theorem \ref{flopinv} based on contour integrals. 

\subsubsection{}
The $k = 1$ case of Theorem \ref{flopinv} says that 
\begin{equation}
    \chi(X(1, n)) = \chi(X^\vee(1, n))
\end{equation}
which, upon spelling out localization formulas, is the equivalent to the elementary identity of rational functions 
\begin{equation} \label{cuteidentity}
    \sum_{i = 1}^n \prod_{j (\neq i)} \frac{1 - ty_j/y_i}{1 - y_j/y_i} \prod_{\ell = 1}^n \frac{1 - tq^{-1}y_i/x_\ell}{1 - q^{-1}y_i/x_\ell} = \sum_{\ell = 1}^n \prod_{m (\neq \ell)} \frac{1 - tx_\ell/x_m}{1 - x_\ell/x_m} \prod_{i = 1}^n \frac{1 - tq^{-1}y_i/x_\ell }{1- q^{-1}y_i/x_\ell}. 
\end{equation}
An instructive way to prove this identity is to study the contour integral 
\begin{equation}
    I = \frac{1}{1 - t}\int_\gamma \frac{dz}{2\pi i z} \prod_{j = 1}^n \frac{1 - ty_j/z}{1 - y_j/z} \prod_{m = 1}^n \frac{1 - tq^{-1}z/x_m}{1 - q^{-1}z/x_m}.
\end{equation}
We pick the the chamber of equivariant variables $|y_1| < \dots < |y_n| < |qx_1| < \dots < |qx_n|$ and we choose the integration contour $\gamma$ such that $|z| = r = \text{const}$ where $|y_j| < r < |qx_m|$ for all $j$ and $m$. The integral should be viewed as over a middle-dimensional cycle in $\mathbb{C}^\times$ with coordinate $z$. The cycle is a certain translate of the maximal compact subgroup $U(1) \subset \mathbb{C}^\times$. 

The integral $I$ can be evaluated in two different ways: by sliding the integration contour towards $|z| \to 0$ or $|z| \to \infty$. For the $|z| \to 0$ evaluation we pick up residues at $z = y_i$ and we have 
\begin{equation}
    I = (\text{LHS of \eqref{cuteidentity}}) + \frac{t^n}{1 - t}
\end{equation}
where the second term is from the residue at $z = 0$. For the $|z| \to \infty$ evaluation we pick up residues at $z = qx_\ell$ and find
\begin{equation}
    I = (\text{RHS of \eqref{cuteidentity}}) + \frac{t^n}{1 - t}
\end{equation}
where the second term is from the residue at $z = \infty$. Noting that the residues at zero and infinity are the same, we conclude \eqref{cuteidentity}. 

\subsubsection{}
Following Section 2 in \cite{ko22}, we can give the following geometric interpretation to the above manipulations. Consider the vector space $V = \mathbb{C}^n
_x \oplus \mathbb{C}^n_y$ as in Section \ref{geometricdiffeq}. We consider the action of $\mathbb{C}^\times$ which scales $\mathbb{C}^n_x$ with weight $-1$ and $\mathbb{C}^n_y$ with weight $+1$. The quotient stack $[V/\mathbb{C}^\times]$ contains both $X(1, n)$ and $X^\vee(1, n)$ as open substacks corresponding to the GIT-stable loci for two different choices of stability conditions. 

The $\chi$-genus of the quotient stack $[V/\mathbb{C}^\times]$ is by definition equal to the integral 
\begin{equation}
    \chi([V/\mathbb{C}^\times]) = \frac{1}{1 - t}\int_\gamma \frac{dz}{2\pi i z} \prod_{j = 1}^n \frac{1 - ty_j/z}{1 - y_j/z} \prod_{m = 1}^n \frac{1 - tq^{-1}z/x_m}{1 - q^{-1}z/x_m} = I
\end{equation}
where the contour $\gamma$ is chosen by requiring that all directions in $V$ are attracted to the origin by the torus action in our chamber of the $y_j$ and $x_m$ variables, see the discussion on attracting lifts of distributions in \cite{ko22}. 

\subsubsection{}
Now let $X$ be any virtually smooth stack which is cohomologically proper for the action of a group $G$ (we refer again to \cite{ko22}, \cite{ko24} for clarifications on the meaning of these terms). Given a $G$-equivariant closed inclusion $Y \xhookrightarrow{} X$ and $G$-equivariant coherent sheaf $\mathscr{F}$ on $X$ we recall there are local cohomology groups $H^i_Y(X; \mathscr{F})$ which fit into a long exact sequence of the form 
\begin{equation}
    \dots \to H^i_Y(X; \mathscr{F}) \to H^i(X; \mathscr{F}) \to H^i(X \setminus Y; \mathscr{F}) \to H^{i + 1}_Y(X ; \mathscr{F}) \to  \dots 
\end{equation}
All terms in this sequence are, of course, $G$-modules in the equivariant situation. Now if $\mathscr{F}$ is taken to be the sheaf of differential forms then the long exact sequence for local cohomology gives the equality of Hirzebruch genera 
\begin{equation}
    \chi(X) = \chi(X \setminus Y) + \chi_Y(X)
\end{equation}
where the third term is by definition the character under $G$ of the local cohomology. This equation just says informally that if $X$ can be cut into pieces then $\chi(X)$ is a sum of contributions from each piece. 

\subsubsection{}
Returning to our concrete situation, there are two distinguished stratifications of $[V/\mathbb{C}^\times]$: 
\begin{equation}
\begin{split}
    [V/\mathbb{C}^\times] & = [(\mathbb{C}^n_x \oplus (\mathbb{C}^n_y \setminus 0))/\mathbb{C}^\times] \, \bigsqcup \, [(\mathbb{C}^n_x \oplus \{ 0 \})/\mathbb{C}^\times] = X(1, n) \, \bigsqcup \, [(\mathbb{C}^n_x \oplus \{ 0 \})/\mathbb{C}^\times] \\
    [V/\mathbb{C}^\times] & = [((\mathbb{C}^n_x \setminus 0 \oplus \mathbb{C}^n_y))/\mathbb{C}^\times] \, \bigsqcup \,  [(\{ 0 \} \oplus \mathbb{C}^n_y)/\mathbb{C}^\times] = X^\vee(1, n) \, \bigsqcup \, [( \{ 0 \} \oplus \mathbb{C}^n_y)/\mathbb{C}^\times].
\end{split}
\end{equation}
These are precisely the stratifications by GIT stable and unstable loci for each choice of stability condition for the $\mathbb{C}^\times$ action on $V$. Then we have 
\begin{equation}
    \chi([V/\mathbb{C}^\times]) = \chi(X(1, n)) + \text{local} = \chi(X^\vee(1, n)) + \text{local}.
\end{equation}
The local pieces are the contributions of the unstable loci, and they may be identified with the residues of the integrand in $I$ at zero at infinity. This is because the character of the local cohomology may be computed by a contour integral with the same integrand as $I$, but the requirement that the contour of integration be placed in such a way that the weights appearing in $N_{X/Y}$ must all be repelling. For the first stratification, this means $|z| < |y_j|$ for all $j$ and for the second, $|z| > |qx_m|$ for all $m$. 
\subsubsection{}
In principle, Theorem \ref{flopinv} could be approached from a similar angle by analyzing the contour integral 
\begin{equation}
    I(k, n) = \frac{1}{k! (1 - t)^k}\int_\gamma \prod_{j = 1}^k \frac{dz_j}{2\pi i z_j} \prod_{i \neq j} \frac{1 - z_i/z_j}{1 - tz_i/z_j} \prod_{\alpha = 1}^n \prod_{i = 1}^k \frac{1 - ty_\alpha/z_i}{1 - y_\alpha/z_i} \prod_{\ell = 1}^n \prod_{i = 1}^k \frac{1 - tq^{-1}z_i/x_\ell}{1 - q^{-1}z_i/x_\ell}. 
\end{equation}
The contour $\gamma$ is a middle-dimensional cycle in the maximal torus of $GL_k$ which is a certain translate of the compact torus. However, actual analysis of the iterated residues becomes more combinatorially involved, which is a direct reflection of the fact that the geometry of $GL_k$-unstable loci for $k > 1$ can be quite intricate. Our approach in Section \ref{geometricdiffeq} avoids this issue by dealing directly with the spaces $X(k, n)$. 

One intuition one gains from the contour integral approach is that for Theorem \ref{flopinv} to work, there needs to be a symmetry between the residues at zero and infinity in the contour integrals. This corresponds to a symmetry between the GIT stratifications for different choices of stability condition. For this to exist, it is essential that $X(k, n)$ satisfies a Calabi-Yau condition. This is the reason the wall-crossing formula for $X(k, n, m)$ in Section \ref{satake} is more complicated. 

\end{appendices}

\newpage

\printbibliography

@misc{kw,
      title={Electric-Magnetic Duality And The Geometric Langlands Program}, 
      author={A. Kapustin and E. Witten},
      year={2007},
      eprint={hep-th/0604151},
      archivePrefix={arXiv},
      primaryClass={hep-th}
}

@misc{bfn,
      title={Towards a mathematical definition of Coulomb branches of $3$-dimensional $\mathcal N=4$ gauge theories, II}, 
      author={A. Braverman and M. Finkelberg and H. Nakajima},
      year={2019},
      eprint={1601.03586},
      archivePrefix={arXiv},
      primaryClass={math.RT}
}

@misc{bfm,
      title={Equivariant ($K$-)homology of affine Grassmannian and Toda lattice}, 
      author={R. Bezrukavnikov and M. Finkelberg and I. Mirković},
      year={2014},
      eprint={math/0306413},
      archivePrefix={arXiv},
      primaryClass={math.AG}
}

@misc{bfnslice,
      title={Coulomb branches of $3d$ $\mathcal N=4$ quiver gauge theories and slices in the affine Grassmannian (with appendices by Alexander Braverman, Michael Finkelberg, Joel Kamnitzer, Ryosuke Kodera, Hiraku Nakajima, Ben Webster, and Alex Weekes)}, 
      author={A. Braverman and M. Finkelberg and H. Nakajima},
      year={2018},
      eprint={1604.03625},
      archivePrefix={arXiv},
      primaryClass={math.RT}
}

@article{Bullimore_2016,
	doi = {10.1007/jhep10(2016)108},
  
	url = {https://doi.org/10.1007%2Fjhep10%282016%29108},
  
	year = 2016,
	month = {oct},
  
	publisher = {Springer Science and Business Media {LLC}
},
  
	volume = {2016},
  
	number = {10},
  
	author = {M. Bullimore and T. Dimofte and D. Gaiotto and J. Hilburn},
  
	title = {Boundaries, mirror symmetry, and symplectic duality in 3d $\mathcal{N} = 4$ gauge theory},
  
	journal = {Journal of High Energy Physics}
}

@misc{teleman2014gauge,
      title={Gauge theory and mirror symmetry}, 
      author={C. Teleman},
      year={2014},
      eprint={1404.6305},
      archivePrefix={arXiv},
      primaryClass={math-ph}
}

@article{frt,
    author = "Faddeev, L. D. and Reshetikhin, N. Yu. and Takhtajan, L. A.",
    title = "{Quantization of Lie Groups and Lie Algebras}",
    reportNumber = "LOMI-E-14-87",
    journal = "Alg. Anal.",
    volume = "1",
    number = "1",
    pages = "178--206",
    year = "1989"
}

@misc{MO,
      title={Quantum Groups and Quantum Cohomology}, 
      author={D. Maulik and A. Okounkov},
      year={2018},
      eprint={1211.1287},
      archivePrefix={arXiv},
      primaryClass={math.AG},
      url={https://arxiv.org/abs/1211.1287}, 
}

@article{Frassek_2022,
   title={Lax matrices from antidominantly shifted Yangians and quantum affine algebras: A-type},
   volume={401},
   ISSN={0001-8708},
   url={http://dx.doi.org/10.1016/j.aim.2022.108283},
   DOI={10.1016/j.aim.2022.108283},
   journal={Advances in Mathematics},
   publisher={Elsevier BV},
   author={Frassek, R. and Pestun, V. and Tsymbaliuk, A.},
   year={2022},
   month=jun, pages={108283} }

@misc{okounkovpcmi,
      title={Lectures on K-theoretic computations in enumerative geometry}, 
      author={A. Okounkov},
      year={2017},
      eprint={1512.07363},
      archivePrefix={arXiv},
      primaryClass={math.AG},
      url={https://arxiv.org/abs/1512.07363}, 
}

@misc{cfkimmaulik,
      title={Stable quasimaps to GIT quotients}, 
      author={I. Ciocan-Fontanine and B. Kim and D. Maulik},
      year={2011},
      eprint={1106.3724},
      archivePrefix={arXiv},
      primaryClass={math.AG},
      url={https://arxiv.org/abs/1106.3724}, 
}

@misc{ko22,
      title={On the unramified Eisenstein spectrum}, 
      author={D. Kazhdan and A. Okounkov},
      year={2022},
      eprint={2203.03486},
      archivePrefix={arXiv},
      primaryClass={math.NT},
      url={https://arxiv.org/abs/2203.03486}, 
}

@misc{ko24,
      title={L-function genera and applications}, 
      author={D. Kazhdan and A. Okounkov},
      year={2024},
      eprint={2311.17747},
      archivePrefix={arXiv},
      primaryClass={math.NT},
      url={https://arxiv.org/abs/2311.17747},
}

@inproceedings{laumon,
    author = {G. Laumon},
    title = {Faisceaux automorphes li´es aux s´eries d’Eisenstein},
    booktitle = {Automorphic forms, Shimura varieties, and L-functions},
    year = {1990},
}

@misc{bravermangaits,
      title={Geometric Eisenstein series}, 
      author={A. Braverman and D. Gaitsgory},
      year={2000},
      eprint={math/9912097},
      archivePrefix={arXiv},
      primaryClass={math.AG},
      url={https://arxiv.org/abs/math/9912097}, 
}

@article{Koroteev_2021,
   title={Quantum K-theory of quiver varieties and many-body systems},
   volume={27},
   ISSN={1420-9020},
   url={http://dx.doi.org/10.1007/s00029-021-00698-3},
   DOI={10.1007/s00029-021-00698-3},
   number={5},
   journal={Selecta Mathematica},
   publisher={Springer Science and Business Media LLC},
   author={Koroteev, P. and Pushkar, P. P. and Smirnov, A. V. and Zeitlin, A. M.},
   year={2021},
   month=aug }

@article{koroteevzeitlin,
   title={qKZ/tRS duality via quantum $K$-theoretic counts},
   volume={28},
   ISSN={1945-001X},
   url={http://dx.doi.org/10.4310/MRL.2021.v28.n2.a5},
   DOI={10.4310/mrl.2021.v28.n2.a5},
   number={2},
   journal={Mathematical Research Letters},
   publisher={International Press of Boston},
   author={Koroteev, P. and Zeitlin, A. M.},
   year={2021},
   pages={435–470} 
}

@article{koroteev_2015,
   title={Defects and quantum Seiberg-Witten geometry},
   volume={2015},
   ISSN={1029-8479},
   url={http://dx.doi.org/10.1007/JHEP05(2015)095},
   DOI={10.1007/jhep05(2015)095},
   number={5},
   journal={Journal of High Energy Physics},
   publisher={Springer Science and Business Media LLC},
   author={Bullimore, M. and Kim, H.-C. and Koroteev, P.},
   year={2015},
   month=may }

@article{Gaiotto_Koroteev,
   title={On three dimensional quiver gauge theories and integrability},
   volume={2013},
   ISSN={1029-8479},
   url={http://dx.doi.org/10.1007/JHEP05(2013)126},
   DOI={10.1007/jhep05(2013)126},
   number={5},
   journal={Journal of High Energy Physics},
   publisher={Springer Science and Business Media LLC},
   author={Gaiotto, D. and Koroteev, P.},
   year={2013},
   month=may }

@misc{tam24, 
title = {Coulomb branches in 3d $\mathscr{N} = 4$ revisited}, 
author = {S. Tamagni}, 
year = {2024},
}

@misc{kapustinht,
      title={Holomorphic reduction of N=2 gauge theories, Wilson-'t Hooft operators, and S-duality}, 
      author={A. Kapustin},
      year={2006},
      eprint={hep-th/0612119},
      archivePrefix={arXiv},
      primaryClass={hep-th},
      url={https://arxiv.org/abs/hep-th/0612119}, 
}

@article{Givental_1995,
   title={Quantum cohomology of flag manifolds and Toda lattices},
   volume={168},
   ISSN={1432-0916},
   url={http://dx.doi.org/10.1007/BF02101846},
   DOI={10.1007/bf02101846},
   number={3},
   journal={Communications in Mathematical Physics},
   publisher={Springer Science and Business Media LLC},
   author={Givental, A. and Kim, B.},
   year={1995},
   month=apr, pages={609–641} }

@misc{braverman2004,
      title={Instanton counting via affine Lie algebras I: Equivariant J-functions of (affine) flag manifolds and Whittaker vectors}, 
      author={A. Braverman},
      year={2004},
      eprint={math/0401409},
      archivePrefix={arXiv},
      primaryClass={math.AG},
      url={https://arxiv.org/abs/math/0401409}, 
}

@misc{gonzalezmakpomerleano,
      title={Coulomb branch algebras via symplectic cohomology}, 
      author={E. Gonzalez and C. Y. Mak and D. Pomerleano},
      year={2023},
      eprint={2305.04387},
      archivePrefix={arXiv},
      primaryClass={math.SG},
      url={https://arxiv.org/abs/2305.04387}, 
}

@article{Cecotti_2014,
   title={tt * geometry in 3 and 4 dimensions},
   volume={2014},
   ISSN={1029-8479},
   url={http://dx.doi.org/10.1007/JHEP05(2014)055},
   DOI={10.1007/jhep05(2014)055},
   number={5},
   journal={Journal of High Energy Physics},
   publisher={Springer Science and Business Media LLC},
   author={Cecotti, S. and Gaiotto, D. and Vafa, C.},
   year={2014},
   month=may }

@misc{gammagehilburn,
      title={Hypertoric 2-categories O and symplectic duality}, 
      author={B. Gammage and J. Hilburn},
      year={2023},
      eprint={2310.06172},
      archivePrefix={arXiv},
      primaryClass={math.RT},
      url={https://arxiv.org/abs/2310.06172}, 
}

@misc{gammagehilburnmazelgee,
      title={Perverse schobers and 3d mirror symmetry}, 
      author={B. Gammage and J. Hilburn and A. Mazel-Gee},
      year={2023},
      eprint={2202.06833},
      archivePrefix={arXiv},
      primaryClass={math.RT},
      url={https://arxiv.org/abs/2202.06833}, 
}

@article{Costello_2020,
   title={Unification of integrability in supersymmetric gauge theories},
   volume={24},
   ISSN={1095-0753},
   url={http://dx.doi.org/10.4310/ATMP.2020.v24.n8.a1},
   DOI={10.4310/atmp.2020.v24.n8.a1},
   number={8},
   journal={Advances in Theoretical and Mathematical Physics},
   publisher={International Press of Boston},
   author={Costello, K. and Yagi, J.},
   year={2020},
   pages={1931–2041} }

@misc{costello2021qoperatorsthooftlines,
      title={Q-operators are 't Hooft lines}, 
      author={K. Costello and D. Gaiotto and J. Yagi},
      year={2021},
      eprint={2103.01835},
      archivePrefix={arXiv},
      primaryClass={hep-th},
      url={https://arxiv.org/abs/2103.01835}, 
}

@misc{aganagicnekrasov,
      author={M. Aganagic},
      year={2024},
      howpublished={private communication},
}

@misc{ferrarizhang,
      title={Difference Equations: from Berry Connections to the Coulomb Branch}, 
      author={A. E. V. Ferrari and D. Zhang},
      year={2024},
      eprint={2409.00173},
      archivePrefix={arXiv},
      primaryClass={hep-th},
      url={https://arxiv.org/abs/2409.00173}, 
}

@article{Lee_2021,
   title={Quantum spin systems and supersymmetric gauge theories. Part I},
   volume={2021},
   ISSN={1029-8479},
   url={http://dx.doi.org/10.1007/JHEP03(2021)093},
   DOI={10.1007/jhep03(2021)093},
   number={3},
   journal={Journal of High Energy Physics},
   publisher={Springer Science and Business Media LLC},
   author={Lee, N. and Nekrasov, N.},
   year={2021},
   month=mar }

@misc{nairtam,
      author={S. Nair and S. Tamagni},
      title={In preparation}
}

@misc{krylovperunov,
      title={Almost dominant generalized slices and convolution diagrams over them}, 
      author={V. Krylov and I. Perunov},
      year={2021},
      eprint={1903.08277},
      archivePrefix={arXiv},
      primaryClass={math.RT},
      url={https://arxiv.org/abs/1903.08277}, 
}

@misc{tamseminar, 
title={Quasimaps with monopoles and nonabelian shift operators}, 
author={S. Tamagni}, 
year={2024}, 
howpublished={\url{https://youtu.be/HdM4TBvTv44?si=8yqCRUctXkiWTBgo}},
}

@misc{afo,
      title={Quantum q-Langlands Correspondence}, 
      author={M. Aganagic and E. Frenkel and A. Okounkov},
      year={2018},
      eprint={1701.03146},
      archivePrefix={arXiv},
      primaryClass={hep-th},
      url={https://arxiv.org/abs/1701.03146}, 
}

@misc{gerasimov2003,
      title={Representation Theory and the Quantum Inverse Scattering Method: The Open Toda Chain and the Hyperbolic Sutherland Model}, 
      author={A. Gerasimov and S. Kharchev and D. Lebedev},
      year={2003},
      eprint={math/0204206},
      archivePrefix={arXiv},
      primaryClass={math.QA},
      url={https://arxiv.org/abs/math/0204206}, 
}

@misc{sklyanin2000,
      title={Baecklund transformations and Baxter's Q-operator}, 
      author={E. K. Sklyanin},
      year={2000},
      eprint={nlin/0009009},
      archivePrefix={arXiv},
      primaryClass={nlin.SI},
      url={https://arxiv.org/abs/nlin/0009009}, 
}

@misc{ginzburg2000,
      title={Perverse sheaves on a Loop group and Langlands' duality}, 
      author={V. Ginzburg},
      year={2000},
      eprint={alg-geom/9511007},
      archivePrefix={arXiv},
      primaryClass={alg-geom},
      url={https://arxiv.org/abs/alg-geom/9511007}, 
}

@misc{mirkovicvilonen,
      title={Geometric Langlands duality and representations of algebraic groups over commutative rings}, 
      author={I. Mirkovic and K. Vilonen},
      year={2018},
      eprint={math/0401222},
      archivePrefix={arXiv},
      primaryClass={math.RT},
      url={https://arxiv.org/abs/math/0401222}, 
}

@misc{beilinsondrinfeld,
    title={Quantization of Hitchin's integrable system and Hecke eigensheaves}, 
    author={A. Beilinson and V. Drinfeld}, 
    year={1991}, 
    howpublished={\url{https://web.ma.utexas.edu/users/benzvi/BD/hitchin.pdf?}},
}

@misc{nakajima2011,
      title={Handsaw quiver varieties and finite W-algebras}, 
      author={H. Nakajima},
      year={2011},
      eprint={1107.5073},
      archivePrefix={arXiv},
      primaryClass={math.QA},
      url={https://arxiv.org/abs/1107.5073}, 
}

@article{Iritani_2009,
   title={An integral structure in quantum cohomology and mirror symmetry for toric orbifolds},
   volume={222},
   ISSN={0001-8708},
   url={http://dx.doi.org/10.1016/j.aim.2009.05.016},
   DOI={10.1016/j.aim.2009.05.016},
   number={3},
   journal={Advances in Mathematics},
   publisher={Elsevier BV},
   author={Iritani, H.},
   year={2009},
   month=oct, pages={1016–1079} }

@misc{molev1994,
      title={Yangians and Classical Lie Algebras}, 
      author={A. Molev and M. Nazarov and G. Olshanskii},
      year={1994},
      eprint={hep-th/9409025},
      archivePrefix={arXiv},
      primaryClass={hep-th},
      url={https://arxiv.org/abs/hep-th/9409025}, 
}

@misc{finkelberg2017,
      title={Comultiplication for shifted Yangians and quantum open Toda lattice}, 
      author={M. Finkelberg and J. Kamnitzer and K. Pham and L. Rybnikov and A. Weekes},
      year={2017},
      eprint={1608.03331},
      archivePrefix={arXiv},
      primaryClass={math.RT},
      url={https://arxiv.org/abs/1608.03331}, 
}

@article{Moore_2000,
   title={Integrating over Higgs Branches},
   volume={209},
   ISSN={0010-3616},
   url={http://dx.doi.org/10.1007/PL00005525},
   DOI={10.1007/pl00005525},
   number={1},
   journal={Communications in Mathematical Physics},
   publisher={Springer Science and Business Media LLC},
   author={Moore, G. and Nekrasov, N. and Shatashvili, S.},
   year={2000},
   month=jan, pages={97–121} }

@misc{liu2024,
      title={Invariance of elliptic genus under wall-crossing}, 
      author={H. Liu},
      year={2024},
      eprint={2405.12587},
      archivePrefix={arXiv},
      primaryClass={math.AG},
      url={https://arxiv.org/abs/2405.12587}, 
}

@misc{borisov2000,
      title={Elliptic Genera of Singular Varieties}, 
      author={L. Borisov and A. Libgober},
      year={2001},
      eprint={math/0007108},
      archivePrefix={arXiv},
      primaryClass={math.AG},
      url={https://arxiv.org/abs/math/0007108}, 
}

@misc{wang2002,
      title={K-equivalence in Birational Geometry}, 
      author={C-L. Wang},
      year={2002},
      eprint={math/0204160},
      archivePrefix={arXiv},
      primaryClass={math.AG},
      url={https://arxiv.org/abs/math/0204160}, 
}

\end{document}